\documentclass[10pt, conference, compsocconf]{IEEEtran}
\usepackage{balance}
\usepackage{amsmath,amssymb, graphicx}
\usepackage[ruled,linesnumbered,vlined,shortend]{algorithm2e}
\usepackage{mdwlist}
\usepackage{color}
\usepackage{multirow}
\usepackage{rotating}
\usepackage{slashbox}
\usepackage{booktabs}
\usepackage{color}
\usepackage[center]{subfigure}
\usepackage{mathrsfs}
\usepackage{amsbsy}
\usepackage[font={small}]{caption}
\usepackage[mathscr]{eucal} 
\usepackage{url}
\let\labelindent\relax
\usepackage{enumitem}
\setitemize{noitemsep,topsep=0pt,parsep=0pt,partopsep=0pt}
\allowdisplaybreaks

\newtheorem{definition}{Definition}

\newtheorem{theorem}{Theorem}

\newcommand{\hide}[1]{}

\newcommand{\bad}[1]{
	\textcolor{red}{#1}
}

\newcommand{\tensor}[1]{\boldsymbol{\mathscr{#1}}}   
\newcommand{\mat}[1]{\mathbf{#1}}
\newcommand{\vect}[1]{\mathbf{#1}}

\newcommand{\argmin}[1]{\underset{#1}{\operatorname{arg}\,\operatorname{min}}\;}

\newcommand{\CD}{\textsc{CDTF}\xspace}
\newcommand{\CA}{\textsc{SALS}\xspace}

\newcommand{\FF}{\textsc{FlexiFaCT}\xspace}
\newcommand{\MR}{\textsc{MapReduce}\xspace}

\begin{document}

\title{
Distributed Methods for High-dimensional and Large-scale Tensor Factorization}

\author{\IEEEauthorblockN{Kijung Shin}
\IEEEauthorblockA{Dept. of Computer Science and Engineering\\
Seoul National University\\
Seoul, Republic of Korea\\
koreaskj@snu.ac.kr}
\and
\IEEEauthorblockN{U Kang}
\IEEEauthorblockA{Dept. of Computer Science\\
KAIST\\
Daejeon, Republic of Korea\\
ukang@cs.kaist.ac.kr}
}

\maketitle

\begin{abstract}
Given a high-dimensional large-scale tensor, how can we decompose it into latent factors?
Can we process it on commodity computers with limited memory?
These questions are closely related to recommender systems, which have modeled rating data not as a matrix but as a tensor to utilize contextual information such as time and location.
This increase in the dimension requires tensor factorization methods scalable with both the dimension and size of a tensor.
In this paper, we propose two distributed tensor factorization methods, \CA and \CD.
Both methods are scalable with all aspects of data, and they show an interesting trade-off between convergence speed and memory requirements.
\CA updates a subset of the columns of a factor matrix at a time, and \CD, a special case of \CA, updates one column at a time. 
In our experiments, only our methods factorize a 5-dimensional tensor with 1 billion observable entries, 10M mode length, and 1K rank, while all other state-of-the-art methods fail.
Moreover, our methods require several orders of magnitude less memory than our competitors.
We implement our methods on \MR with two widely-applicable optimization techniques: local disk caching and greedy row assignment.
They speed up our methods up to 98.2$\times$ and also the competitors up to 5.9$\times$. 
\end{abstract}


\section{Introduction and Related Work}
\label{sec:intro}
Recommendation problems can be viewed as completing a partially observable user-item matrix whose entries are ratings. Matrix factorization (MF), which decomposes the input matrix into a user factor matrix and an item factor matrix such that their multiplication approximates the input matrix, is one of the most widely-used methods for matrix completion \cite{chen2011linear, koren2009matrix, zhou2008large}.
To handle web-scale data, efforts were made to find distributed ways for MF, including ALS \cite{zhou2008large}, DSGD \cite{gemulla2011large}, and CCD++ \cite{yu2013parallel}.

On the other hand, there have been attempts to improve the accuracy of recommendation by using additional contextual information such as time and location.
A straightforward way to utilize such extra factors is to model rating data as a partially observable tensor where additional dimensions correspond to the extra factors.
Similar to the matrix completion, tensor factorization (TF), which decomposes the input tensor into multiple factor matrices and a core tensor, has been used for tensor completion \cite{Karatzoglou:2010, nanopoulos2010musicbox, zheng2010collaborative}.

\begin{table}[tbp!]
\centering
\caption{Summary of scalability results. The factors which each method is scalable with are checked. Our proposed \CA and \CD are the only methods scalable with all the factors.}
\begin{tabular}{c|ccccc}
\toprule
\textbf{} & \textbf{\CD} & \textbf{\CA} & ALS & PSGD & \FF \\
\midrule
  Dimension & {\large \checkmark} & {\large \checkmark} & \checkmark & \checkmark &  \\
  Observations & {\large \checkmark} & {\large \checkmark} & \checkmark & \checkmark & \checkmark \\
  Mode Length & {\large \checkmark} &  {\large \checkmark} & &  & \checkmark \\
  Rank & {\large \checkmark} & {\large \checkmark} & &  & \checkmark \\
  Machines & {\large \checkmark} & {\large \checkmark} & \checkmark & & \\
\bottomrule
\end{tabular}
\label{tab:scale}
\end{table}

As the dimension of web-scale recommendation problems increases, a necessity for TF algorithms scalable with the dimension as well as the size of data has arisen.
A promising way to find such algorithms is to extend distributed MF algorithms to higher dimensions.
However, the scalability of existing methods including ALS \cite{zhou2008large}, PSGD \cite{mcdonald2010distributed}, and \FF \cite{beutel2014flexifact} is limited as summarized in Table~\ref{tab:scale}.

In this paper, we propose Subset Alternating Least Square (\CA) and Coordinate Descent for Tensor Factorization (\CD), distributed tensor factorization methods scalable with all aspects of data.
\CA updates a subset of the columns of a factor matrix at a time, and \CD, a special case of \CA, updates one column at a time. 
These two methods have distinct advantages: \CA converges faster, and \CD is more memory-efficient.
Our methods can also be used in any applications handling large-scale partially observable tensors, including social network analysis \cite{dunlavy2011temporal} and Web search \cite{sun2005cubesvd}.

\begin{table*}[tbp!]
\vspace{-2mm}
\caption{Summary of distributed tensor factorization algorithms. The performance bottlenecks which prevent each algorithm from handling web-scale data are colored red.
Only our proposed \CA and \CD have no bottleneck.
Communication complexity is measured by the number of parameters that each machine exchanges with the others.
For simplicity, we assume that workload of each algorithm is equally distributed across machines, that the length of every mode is equal to $I$, and that $T_{in}$ of \CA and \CD is set to one.}
\centering
\begin{tabular}{c|cccc}
\toprule
        \textbf{Algorithm} & \textbf{Computational complexity} & \textbf{Communication complexity} & \textbf{Memory requirements} & \textbf{Convergence speed} \\
         & (per iteration) & (per iteration) &  &  \\
\midrule
 \textbf{\CD} & $O(|\Omega|N^{2}K/M)$ & $O(NIK)$ & $O(NI)$ & Fast \\
 \textbf{\CA} & $O(|\Omega|NK(N+C)/M+NIKC^{2}/M)$ & $O(NIK)$ & $O(NIC)$ & Fastest \\
 ALS \cite{zhou2008large} & \bad{$O(|\Omega|NK(N+K)/M+NIK^{3}/M)$} & $O(NIK)$ & \bad{$O(NIK)$} & Fastest \\
 PSGD \cite{mcdonald2010distributed} & $O(|\Omega|NK/M)$ & $O(NIK)$ & \bad{$O(NIK)$} & \bad{Slow} \\
 \FF  \cite{beutel2014flexifact}& $O(|\Omega|NK/M)$ & \bad{$O(M^{N-2}NIK)$} & $O(NIK/M)$ & Fast \\
\bottomrule
\end{tabular}
\label{tab:algorithms}
\end{table*}

The main contributions of our study are as follows:
\begin{itemize*}
  \item \textbf{Algorithm.}
  We propose \CA and \CD, scalable tensor factorization algorithms.
  Their distributed versions are the only methods scalable with all the following factors: the dimension and size of data, the number of parameters, and the number of machines (Table \ref{tab:scale}).
  \item \textbf{Analysis.} We analyze our methods and the competitors in a general N-dimensional setting in the following aspects: computational complexity, communication complexity, memory requirements, and convergence speed (Table \ref{tab:algorithms}).
  \item \textbf{Optimization.}
  We implement our methods on \MR with two widely-applicable optimization techniques: local disk caching and greedy row assignment.
  They speed up not only our methods (up to 98.2$\times$) but also the competitors (up to 5.9$\times$) (Figure~\ref{fig:impl}).
  \item \textbf{Experiment.}
  We empirically confirm the superior scalability of our methods and their several orders of magnitude less memory requirements than the competitors.
  Only our methods analyze a 5-dimensional tensor with 1 billion observable entries, 10M mode length, and 1K rank, while all others fail (Figure~\ref{fig:data_scale_overall}).
\end{itemize*}

\begin{table}[t!]
        \centering
        \caption{Table of symbols.}
        \begin{tabular}{cl}
\toprule
        \textbf{Symbol} & \textbf{Definition} \\
\midrule
  $\tensor{X}$ & input tensor $(\in \mathbb{R}^{I_{1} \times I_{2} ... \times I_{N}})$\\
  $x_{i_{1}...i_{N}}$ & $(i_{1},...,i_{N})$th entry of $\tensor{X}$\\
  $N$ & dimension of $\tensor{X}$\\
  $I_{n}$ & length of the $n$th mode of $\tensor{X}$\\
  $\mat{A}^{(n)}$ & $n$th factor matrix $(\in \mathbb{R}^{I_{n} \times K})$ \\
  $a^{(n)}_{i_{n}k}$ & $(i_{n},k)$th entry of $\mat{A}^{(n)}$\\
  $K$ & rank of $\tensor{X}$\\
  $\Omega$ & set of indices of observable entries of $\tensor{X}$\\
  $\Omega^{(n)}_{i_{n}}$ & subset of $\Omega$ whose $n$th mode's index is equal to $i_{n}$ \\
  ${}_mS_{n}$ & set of rows of $\mat{A}^{(n)}$ assigned to machine $m$\\
  $\tensor{R}$ & residual tensor $(\in \mathbb{R}^{I_{1} \times I_{2} ... \times I_{N}})$\\
  $r_{i_{1}...i_{N}}$ & $(i_{1},...,i_{N})$th entry of $\tensor{R}$\\
  $M$ & number of machines (reducers on \MR)\\
  $T_{out}$ & number of outer iterations\\
  $T_{in}$ & number of inner iterations\\
  $\lambda$ & regularization parameter\\
  $C$ & number of parameters updated at a time\\
  $\eta_{0}$ & initial learning rate\\
\bottomrule
\end{tabular}
\label{tab:Symbols}
\end{table}

The binary codes of our methods and several datasets are available at \url{http://kdmlab.org/sals}.
The rest of this paper is organized as follows.
Section~\ref{sec:prelim} presents preliminaries for tensor factorization and its distributed algorithms.
Section~\ref{sec:method} describes our proposed \CA and \CD methods.
Section~\ref{sec:implementation} presents the optimization techniques for them on \MR.
After providing experimental results in Section~\ref{sec:experiment}, we conclude in Section~\ref{sec:conclusion}.

\section{Preliminaries: Tensor Factorization}
\label{sec:prelim}
In this section, we describe the preliminaries of tensor factorization and its distributed algorithms.

\subsection{Tensor and the Notations}
\label{sec:prelim:tensor}
Tensors are multi-dimensional arrays that generalize vectors ($1$-dimensional tensors) and matrices ($2$-dimensional tensors) to higher dimensions.
Like rows and columns in a matrix, an $N$-dimensional tensor has $N$ modes whose lengths are denoted by $I_{1}$ through $I_{N}$, respectively.
We denote tensors with variable dimension $N$ by boldface Euler script letters, e.g., $\tensor{X}$.
Matrices and vectors are denoted by boldface capitals, e.g., $\mat{A}$, and boldface lowercases, e.g., $\vect{a}$, respectively.
We denote the entry of a tensor by the symbolic name of the tensor with its indices in subscript.
For example, the $(i_{1},i_{2})$th entry of $\mat{A}$ is denoted by $a_{i_{1}i_{2}}$, and the $(i_{1},...,i_{N})$th entry of $\tensor{X}$ is denoted by $x_{i_{1}...i_{N}}$.
The $i_{1}$th row of $\mat{A}$ is denoted by $\vect{a}_{i_{1}*}$, and the $i_{2}$th column of $\mat{A}$ is denoted by $\vect{a}_{*i_{2}}$.
Table \ref{tab:Symbols} lists the symbols used in this paper.

\subsection{Tensor Factorization}
\label{sec:prelim:tc}
There are several ways to define a tensor factorization. Our definition is based on the PARAFAC decomposition \cite{kolda2009tensor}, which is one of the most popular decomposition methods, and the nonzero squared loss with $L_{2}$ regularization, whose weighted form has been successfully used in many recommender systems \cite{chen2011linear, koren2009matrix, zhou2008large}.

\begin{definition}[Tensor Factorization]
\ \\ Given an $N$-dimensional tensor $\tensor{X}(\in \mathbb{R}^{I_{1} \times I_{2} ... \times I_{N}})$ with observable entries $\{x_{i_{1}...i_{N}}|(i_{1},...,i_{N})\in \Omega\}$,
the rank $K$ factorization of $\tensor{X}$ is to find factor matrices $\{\mathbf{A}^{(n)}\in \mathbb{R}^{I_{n} \times K}|1 \leq n \leq N\}$ which minimize the following loss function:
{\setlength\arraycolsep{0.1em}
\small
\begin{multline}
L(\mat{A}^{(1)},...,\mat{A}^{(N)}) = \\
\sum_{(i_{1},...,i_{N})\in \Omega}{\left(x_{i_{1}...i_{N}}-\sum_{k=1}^{K}\prod_{n=1}^{N}a^{(n)}_{i_{n}k}\right)^{2}} + \lambda\sum_{n=1}^{N}{{\| \mat{A}^{(n)} \|}_{F}^{2}}
\label{eq:TF}
\end{multline}
}
Note that the loss function depends only on the observable entries.
Each factor matrix $\mathbf{A}^{(n)}$ corresponds to the latent feature vectors of the objects that the $n$th mode of $\tensor{X}$ represents, and $\sum_{k=1}^{K}\prod_{n=1}^{N}a^{(n)}_{i_{n}k}$ corresponds to the interaction among the features.
\end{definition}

\subsection{Distributed Methods for Tensor Factorization}
\label{sec:prelim:tc_alg}
In this section,  we explain how widely-used distributed optimization methods are applied to tensor factorization.
Their performances are summarized in Table~\ref{tab:algorithms}.
Note that only our proposed CDTF and SALS methods, which will be described in Sections~\ref{sec:method} and~\ref{sec:implementation}, have no bottleneck in any aspects.

\subsubsection{Alternating Least Square (ALS)}
\label{sec:prelim:tensor:als}
Using ALS \cite{zhou2008large}, we update factor matrices one by one while keeping all other matrices fixed.
When all other factor matrices are fixed, minimizing \eqref{eq:TF} is analytically solvable in terms of the updated matrix, which can be updated row by row due to the independence between rows.
The update rule for each row of $\mat{A}^{(n)}$ is as follows:
{\small
	\begin{align}
	[a^{(n)}_{i_{n}1}, ..., a^{(n)}_{i_{n}K}]^{T} & \leftarrow \argmin{[a^{(n)}_{i_{n}1}, ..., a^{(n)}_{i_{n}K}]^{T}}{L(\mat{A}^{(1)},...,\mat{A}^{(N)})} \nonumber\\
	& = (\mat{B}^{(n)}_{i_{n}}+\lambda \mat{I}_{K})^{-1}\vect{c}^{(n)}_{i_{n}}
	\label{eq:ALS}
	\end{align}
}
where the $(k_{1},k_{2})$th entry of $\mat{B}^{(n)}_{i_{n}}(\in\mathbb{R}^{K\times K})$ is \\
{\small
	\begin{equation}
	\sum_{(i_{1},...,i_{N})\in \Omega^{(n)}_{i_{n}}}\left(\prod_{l\neq n}a^{(l)}_{i_{l}k_{1}}\prod_{l\neq n}a^{(l)}_{i_{l}k_{2}}
	\right),\nonumber
	\end{equation}
}
{\small
	the $k$th entry of $\vect{c}^{(n)}_{i_{n}}(\in\mathbb{R}^{K})$ is \\
	\begin{equation}
	\sum_{(i_{1},...,i_{N})\in \Omega^{(n)}_{i_{n}}}\left(x_{i_{1}...i_{N}}\prod_{l\neq n}a^{(l)}_{i_{l}k}
	\right), \nonumber
	\end{equation}
} and $\mat{I}_K$ is the $K$ by $K$ identity matrix.
$\Omega^{(n)}_{i_{n}}$ denotes the subset of $\Omega$ whose $n$th mode's index is $i_{n}$.
This update rule can be proven as in Theorem~\ref{thm:SALS:correctness} in Section~\ref{sec:method:update} since ALS is a special case of \CA.
Updating a row, $\vect{a}^{(n)}_{i_{n}*}$ for example, using \eqref{eq:ALS} takes $O(|\Omega_{i_{n}}|K(N+K)+K^{3})$, which consists of $O(|\Omega^{(n)}_{i_{n}}|NK)$ to calculate $\prod_{l\neq n}a^{(l)}_{i_{l}1}$ through $\prod_{l\neq n}a^{(l)}_{i_{l}K}$ for all the entries in $\Omega^{(n)}_{i_{n}}$, $O(|\Omega^{(n)}_{i_{n}}|K^{2})$ to build $\mat{B}_{i_{n}}^{(n)}$, $O(|\Omega^{(n)}_{i_{n}}|K)$ to build $\vect{c}^{(n)}_{i_{n}}$, and $O(K^{3})$ to invert $\mat{B}_{i_{n}}^{(n)}$.
Thus, updating every row of every factor matrix once, which corresponds to a full ALS iteration, takes $O(|\Omega|NK(N+K)+K^{3}\sum_{n=1}^{N}I_{n})$.

In distributed environments, updating each factor matrix can be parallelized without affecting the correctness of ALS by distributing the rows of the factor matrix across machines and updating them simultaneously.
The parameters updated by each machine are broadcast to all other machines. 
The number of parameters each machine exchanges with the others is $O(KI_{n})$ for each factor matrix $\mat{A}^{(n)}$ and $O(K\sum_{n=1}^{N}I_{n})$ per iteration.
The memory requirements of ALS, however, cannot be distributed.
Since the update rule \eqref{eq:ALS} possibly depends on any entry of any fixed factor matrix, every machine is required to load all the fixed matrices into its memory.
This high memory requirements of ALS, $O(K\sum_{n=1}^{N}I_{n})$ memory space per machine, have been noted as a scalability bottleneck even in matrix factorization \cite{gemulla2011large,yu2013parallel}.

\subsubsection{Parallelized Stochastic Gradient Descent (PSGD)}
PSGD \cite{mcdonald2010distributed} is a distributed algorithm based on stochastic gradient descent (SGD).
In PSGD, the observable entries of $\tensor{X}$ are randomly divided into $M$ machines which run SGD independently using the assigned entries.
The updated parameters are averaged after each iteration.
For each observable entry $x_{i_{1}...i_{N}}$, $a^{(n)}_{i_{n}k}$ for all $n$ and $k$, whose number is $NK$, are updated at once by the following rule :
{\small
	\begin{equation}
	{a}^{(n)}_{i_{n}k} \leftarrow a^{(n)}_{i_{n}k} - 2\eta\left(\frac{\lambda a^{(n)}_{i_{n}k}}{|\Omega^{(n)}_{i_{n}}|} - r_{i_{1}...i_{N}}\left(\prod_{l\neq n}{a^{(l)}_{i_{l}k}}\right)\right)
	\label{eq:SGD}
	\end{equation}
}
where
$r_{i_{1}...i_{N}}=x_{i_{1}...i_{N}}-\sum_{s=1}^{K}\prod_{l=1}^{N}{a^{(l)}_{i_{l}s}}$.
It takes $O(NK)$ to calculate $r_{i_{1}...i_{N}}$ and $\prod_{l=1}^{N}{a^{(l)}_{i_{l}k}}$ for all $k$.
Once they are calculated, since $\prod_{l\neq n}{a^{(l)}_{i_{l}k}}$ can be calculated as $\left(\prod_{l=1}^{N}{a^{(l)}_{i_{l}k}}\right)/a^{(n)}_{i_{n}k}$, calculating \eqref{eq:SGD} takes $O(1)$, and thus updating all the $NK$ parameters takes $O(NK)$. 
If we assume that $\tensor{X}$ entries are equally distributed across machines, 
the computational complexity per iteration is $O(|\Omega|NK/M)$.
Averaging parameters can also be distributed, and in the process, $O(K\sum_{n=1}^{N}I_{n})$ parameters are exchanged by each machine.
Like ALS, the memory requirements of PSGD cannot be distributed, i.e., all the machines are required to load all the factor matrices into their memory. $O(K\sum_{n=1}^{N}I_{n})$ memory space is required per machine.
Moreover, PSGD tends to converge slowly due to the non-identifiability of \eqref{eq:TF} \cite{gemulla2011large}.

\subsubsection{Flexible Factorization of Coupled Tensors (FlexiFaCT)}
\FF \cite{beutel2014flexifact} is another SGD-based algorithm that remedies the high memory requirements and slow convergence of PSGD.
\FF divides $\tensor{X}$ into $M^{N}$ blocks. Each $M$ disjoint blocks that do not share common fibers (rows in a general $n$th mode) compose a stratum. 
\FF processes $\tensor{X}$ one stratum at a time in which the $M$ blocks composing a stratum are distributed across machines and processed independently.
The update rule is the same as \eqref{eq:SGD}, and the computational complexity per iteration is $O(|\Omega|NK/M)$ as in PSGD.
However, contrary to PSGD, averaging is unnecessary because a set of parameters updated by each machine are disjoint with those updated by the other machines.
In addition, the memory requirements of \FF are distributed among the machines.
Each machine only needs to load the parameters related to the block it processes, whose number is $(K\sum_{n=1}^{N}I_{n})/M$, into its memory at a time.
However, \FF suffers from high communication cost.
After processing one stratum, each machine sends the updated parameters to the machine which updates them using the next stratum.
Each machine exchanges at most $(K\sum_{n=2}^{N}I_{n})/M$ parameters per stratum and $M^{N-2}K\sum_{n=2}^{N}I_{n}$ per iteration where $M^{N-1}$ is the number of strata.
Thus, the communication cost increases exponentially with the dimension of $\tensor{X}$ and polynomially with the number of machines.


\section{Proposed methods}
\label{sec:method}
In this section, we propose Subset Alternating Least Square (\CA) and Coordinate Descent for Tensor Factorization (\CD).
They are scalable algorithms for tensor factorization, which is essentially an optimization problem whose loss function is \eqref{eq:TF} and parameters are the entries of factor matrices, $\mat{A}^{(1)}$ through $\mat{A}^{(N)}$.
Figure~\ref{fig:SALS} depicts the difference among \CD, \CA, and ALS.
Unlike ALS, which updates each $K$ columns of factor matrices row by row, \CA updates each $C$ $(1\leq C\leq K)$ columns row by row, and \CD updates each column entry by entry.
\CD can be seen as an extension of CCD++ \cite{yu2013parallel} to higher dimensions.
Since \CA contains \CD($C=1$) as well as ALS ($T_{in}=1,C=K$) as a special case, we focus on \CA then explain additional optimization schemes for \CD.
\begin{figure}[t!]
    \centering
    \subfigure[\CD]
    {
    	\includegraphics[width=.20\linewidth]{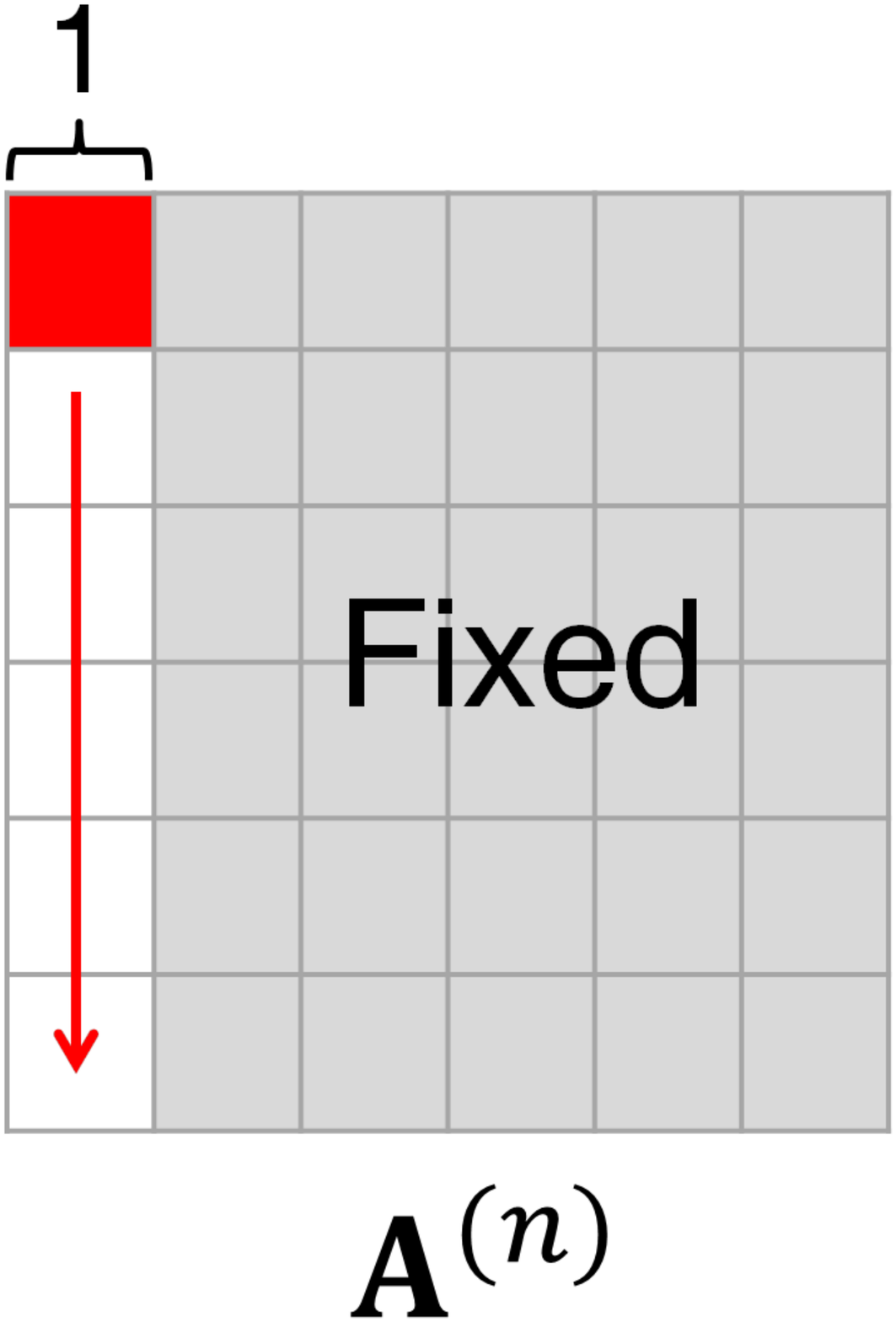}
    }
    \subfigure[\CA]
    {
    	\includegraphics[width=.20\linewidth]{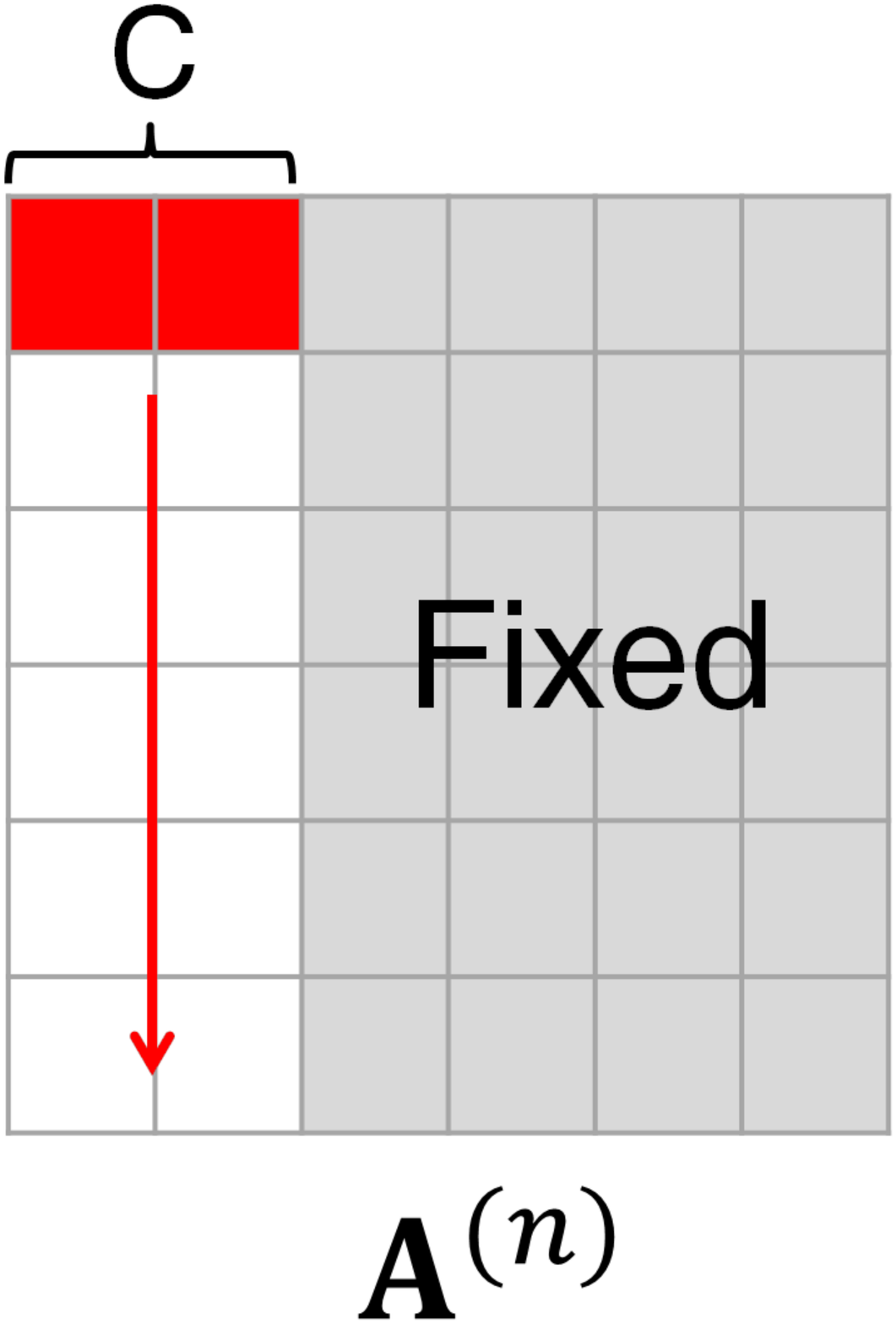}
    }
    \subfigure[ALS]
    {
    	\includegraphics[width=.20\linewidth]{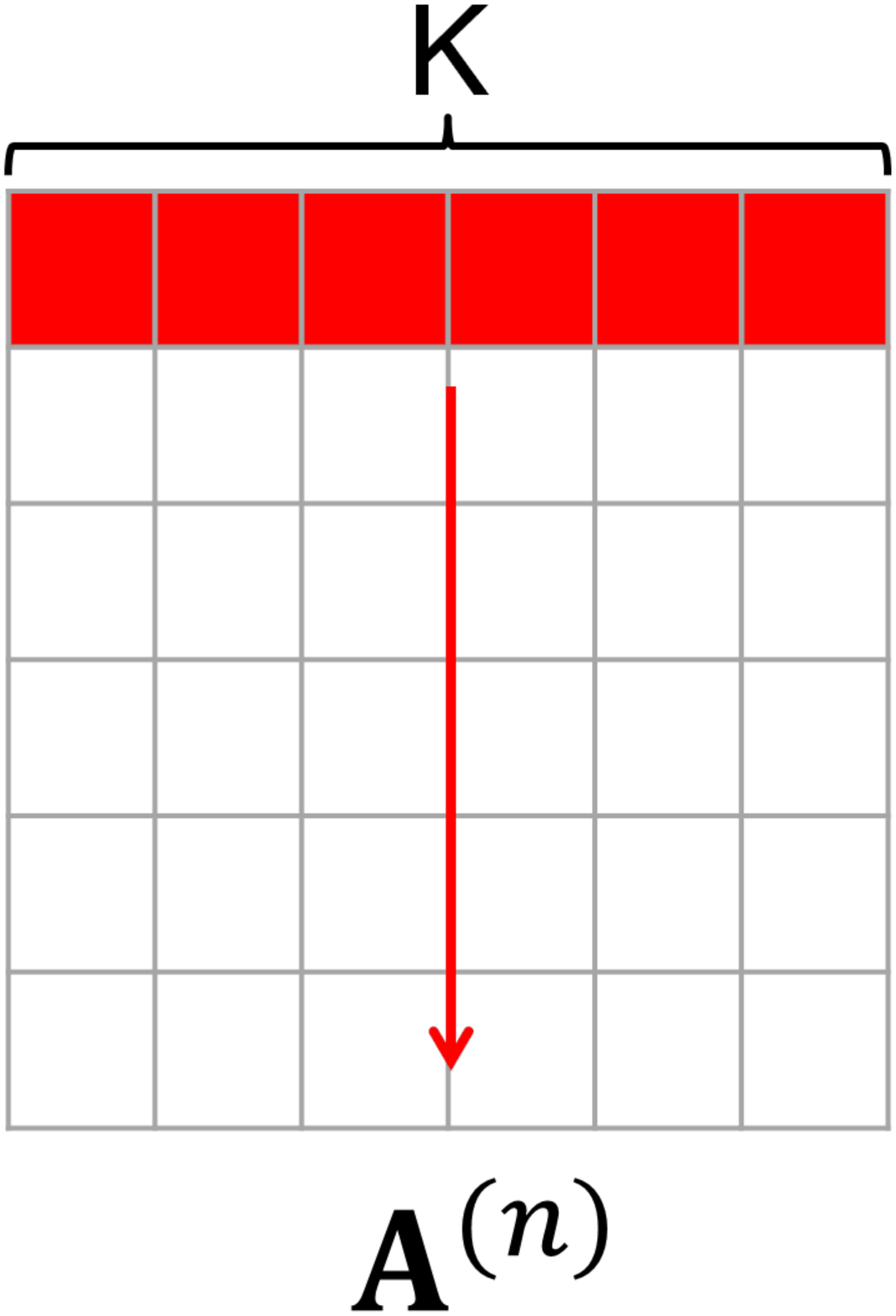}
    }
    \caption{\label{fig:SALS}Update rules of \CD, \CA, and ALS. \CD updates each column of factor matrices entry by entry, \CA updates each $C$ $(1\leq C\leq K)$ columns row by row, and ALS updates each $K$ columns row by row.}
\end{figure}

\subsection{Update Rule and Update Sequence}
\label{sec:method:update}
\begin{algorithm} [tbp!]
\small
\caption{Serial version of \CA} \label{alg:SALS}
  \SetKwInOut{Input}{Input}
  \SetKwInOut{Output}{Output}
  \Input{
  $\tensor{X}$, $K$, $\lambda$
  }
  \Output{
    $\mat{A}^{(n)}$ for all $n$
  }
  \vspace{1.5mm}
  initialize $\tensor{R}$ and $\mat{A}^{(n)}$ for all $n$ \label{alg:SALS:init}\\
  \For{outer iter = 1..$T_{out}$\label{alg:SALS:oter}} {
  	\For{split iter = 1..$\lceil \frac{K}{C} \rceil$}{
  		choose $k_{1}, ..., k_{C}$ {\footnotesize(from columns not updated yet)} \label{alg:SALS:sample}\\
	  	compute $\tensor{\hat{R}}$ \label{alg:SALS:rhat} \\
	  	\For{inner iter = 1..$T_{in}$}{ \label{alg:SALS:Tin}
	  		\For{n = 1..$N$}{ \label{alg:SALS:rankC}
	  			\For{$i_{n}$ = 1..$I_{n}$}{ \label{alg:SALS:column}
	  	   			update $a^{(n)}_{i_{n}k_{1}}, ..., a^{(n)}_{i_{n}k_{C}}$ using \eqref{eq:SALS} \label{alg:SALS:entry}\\
	  	    	}
	  	   	}
	  	 }
	  	 update $\tensor{R}$ using \eqref{eq:SALS_R} \label{alg:SALS:R} \\
	  	}
	 }
\end{algorithm}

Algorithm~\ref{alg:SALS} describes the procedure of \CA.
$\tensor{R}$ denotes the residual tensor where $r_{i_{1}...i_{N}}=x_{i_{1}...i_{N}}-\sum_{k=1}^{K}\prod_{n=1}^{N}{a^{(n)}_{i_{n}k}}$.
We initialize the entries of $\mat{A}^{(1)}$ to zeros and those of all other factor matrices to random values so that the initial value of $\tensor{R}$ is equal to $\tensor{X}$ (line \ref{alg:SALS:init}).
In every iteration (line~\ref{alg:SALS:oter}), \CA repeats choosing $C$ columns, $k_{1}$ through $k_{C}$, randomly without replacement (line~\ref{alg:SALS:sample}) and updating them while keeping the other columns fixed, which is equivalent to the rank $C$ factorization of $\tensor{\hat{R}}$ where $\hat{r}_{i_{1}...i_{N}}=r_{i_{1}...i_{N}}+\sum_{c=1}^{C}\prod_{n=1}^{N}a^{(n)}_{i_{n}k_{c}}$.
Once $\tensor{\hat{R}}$ is computed (line~\ref{alg:SALS:rhat}), updating $C$ columns of factor matrices matrix by matrix (line~\ref{alg:SALS:rankC}) is repeated $T_{in}$ times (line~\ref{alg:SALS:Tin}).
For each factor matrix, since its rows are independent of each other in minimizing \eqref{eq:TF} when the other factor matrices are fixed, the entries are updated row by row (line~\ref{alg:SALS:column}) as follows:
{\small
\begin{align}
 [a^{(n)}_{i_{n}k_{1}}, ..., a^{(n)}_{i_{n}k_{C}}]^{T} \leftarrow & \argmin{[a^{(n)}_{i_{n}k_{1}}, ..., a^{(n)}_{i_{n}k_{C}}]^{T}}{L(\mat{A}^{(1)},...,\mat{A}^{(N)})}\nonumber\\
 & = (\mat{B}^{(n)}_{i_{n}}+\lambda \mat{I}_{C})^{-1}\vect{c}^{(n)}_{i_{n}},
 \label{eq:SALS}
\end{align}
}
where the $(c_{1},c_{2})$th entry of $\mat{B}^{(n)}_{i_{n}}(\in\mathbb{R}^{C\times C})$ is \\
{\footnotesize
\begin{equation}
\sum_{(i_{1},...,i_{N})\in \Omega^{(n)}_{i_{n}}}\left(\prod_{l\neq n}a^{(l)}_{i_{l}k_{c_{1}}}\prod_{l\neq n}a^{(l)}_{i_{l}k_{c_{2}}}
\right),\nonumber
\end{equation}
}
the $c$th entry of $\vect{c}^{(n)}_{i_{n}}(\in\mathbb{R}^{C})$ is \\
{\footnotesize
\begin{equation}
\sum_{(i_{1},...,i_{N})\in \Omega^{(n)}_{i_{n}}}\left(\hat{r}_{i_{1}...i_{N}}\prod_{l\neq n}a^{(l)}_{i_{l}k_{c}}
\right), \nonumber
\end{equation}
}and $\mat{I}_C$ is the $C$ by $C$ identity matrix. $\Omega^{(n)}_{i_{n}}$ denotes the subset of $\Omega$ whose $n$th mode's index is $i_{n}$.
The proof of this update rule is as follows:
\begin{theorem}
	\label{thm:SALS:correctness}
	{\setlength\arraycolsep{0.1em}
		\small
		\begin{equation*}
		\argmin{[a^{(n)}_{i_{n}k_{1}}, ..., a^{(n)}_{i_{n}k_{C}}]^{T}}{L(\mat{A}^{(1)},...,\mat{A}^{(N)})}\nonumber
		= (\mat{B}^{(n)}_{i_{n}}+\lambda \mat{I}_{C})^{-1}\vect{c}^{(n)}_{i_{n}}
		\end{equation*}
	}
\end{theorem}
\begin{proof}
{\setlength\arraycolsep{0.1em}
	\small
	\begin{align*}
	\ & \frac{\partial L}{\partial a^{(n)}_{i_{n}k_{c}}}=0, \forall c, 1 \leq c \leq C \\
	\ & \Leftrightarrow \sum_{(i_{1},...,i_{N})\in \Omega^{n}_{i_{n}}}{\left(\left(\sum_{s=1}^{C}\prod_{l=1} ^{N}a^{(l)}_{i_{l}k_{s}}-\hat{r}_{i_{1}...i_{N}}\right)\prod_{l\neq n}a^{(l)}_{i_{l}k_{c}}\right)}
	+ \lambda a^{(n)}_{i_{n}k_{c}}\\
	& \;\;\;=0, \forall c \\
	\ & \Leftrightarrow \sum_{(i_{1},...,i_{N})\in \Omega^{n}_{i_{n}}}{\left(\sum_{s=1}^{C}\left(a^{(n)}_{i_{n}k_{s}}\prod_{l\neq n}a^{(l)}_{i_{l}k_{s}}\right)\prod_{l\neq n}a^{(l)}_{i_{l}k_{c}}\right)}
	+ \lambda a^{(n)}_{i_{n}k_{c}}\\
	\ & \;\;\; =\sum_{(i_{1},...,i_{N})\in \Omega^{n}_{i_{n}}}\left(\hat{r}_{i_{1}...i_{N}}\prod_{l\neq n}a^{(l)}_{i_{l}k_{c}}\right), \forall c \\
	\ & \Leftrightarrow (\mat{B}^{(n)}_{i_{n}}+\lambda \mat{I}_{c})[a^{(n)}_{i_{n}k_{1}}, ..., a^{(n)}_{i_{n}k_{C}}]^{T}=\vect{c}^{(n)}_{i_{n}}
	\end{align*}
}
\end{proof}
Since $\mat{B}^{(n)}_{i_{n}}$ is symmetric, instead of computing its inverse, the Cholesky decomposition can be used.
After this rank $C$ factorization, the entries of $\tensor{R}$ are updated by the following rule (line~\ref{alg:SALS:R}):
\begin{equation}\small
r_{i_{1}...i_{N}} \leftarrow \hat{r}_{i_{1}...i_{N}}-\sum_{c=1}^{C}\prod_{n=1}^{N}a^{(n)}_{i_{n}k_{c}}.
\label{eq:SALS_R}
\end{equation}

\CD is a special case of \CA where $C$ is set to one.
In \CD, instead of computing $\tensor{\hat{R}}$ before rank one factorization, the entries of $\tensor{\hat{R}}$ can be computed while computing \eqref{eq:SALS} and \eqref{eq:SALS_R}.
This can result in better performance on a disk-based system like \MR by reducing disk I/O operations.
Moreover, instead of randomly changing the order by which columns are updated at each iteration, fixing the order speeds up the convergence of \CD in our experiments.

\subsection{Complexity Analysis}
\begin{theorem}
The computational complexity of Algorithm \ref{alg:SALS} is $O(T_{out}|\Omega|NT_{in}K(N+C)+T_{out}T_{in}KC^{2}\sum_{n=1}^{N}I_{n})$.
\end{theorem}
\begin{proof}

Computing $\tensor{\hat{R}}$ (line~\ref{alg:SALS:rhat}) and updating $\tensor{R}$ (line~\ref{alg:SALS:R}) take $O(|\Omega|NC)$.
Updating $C$ parameters (line~\ref{alg:SALS:entry}) takes $O(|\Omega^{(n)}_{i_{n}}|C(C+N)+C^{3})$, which consists of $O(|\Omega^{(n)}_{i_{n}}|NC)$ to calculate $\prod_{l\neq n}a^{(l)}_{i_{l}k_{1}}$ through $\prod_{l\neq n}a^{(l)}_{i_{l}k_{C}}$ for all the entries in $|\Omega|^{(n)}_{i_{n}}$, $O(|\Omega^{(n)}_{i_{n}}|C^{2})$ to build $\mat{B}_{i_{n}}^{(n)}$, $O(|\Omega^{(n)}_{i_{n}}|C)$ to build $\vect{c}^{(n)}_{i_{n}}$, and $O(C^{3})$ to invert $\mat{B}_{i_{n}}^{(n)}$.
Thus, updating all the entries in $C$ columns (lines~\ref{alg:SALS:column} through \ref{alg:SALS:entry}) takes $O(|\Omega|C(C+N)+I_{n}C^{3})$, and the rank $C$ factorization (lines~\ref{alg:SALS:rankC} through \ref{alg:SALS:entry}) takes $O(|\Omega|NC(N+C)+C^{3}\sum_{n=1}^{N}I_{n})$.
As a result, an outer iteration, which repeats the rank $C$ factorization $T_{in}K/C$ times and both $\tensor{\hat{R}}$ computation and $\tensor{R}$ update $K/C$ times, takes $O(|\Omega|NT_{in}K(N+C)+T_{in}KC^{2}\sum_{n=1}^{N}I_{n})+O(|\Omega|NK)$, where the second term is dominated.
\end{proof}

\begin{theorem}\label{theorem:memory}
The memory requirement of Algorithm~\ref{alg:SALS} is $O(C\sum_{n=1}^{N}I_{n})$.
\end{theorem}
\begin{proof}

Since $\tensor{\hat{R}}$ computation (line~\ref{alg:SALS:rhat}), rank $C$ factorization (lines~\ref{alg:SALS:rankC} through \ref{alg:SALS:entry}), and $\tensor{R}$ update (line~\ref{alg:SALS:R}) all depend only on the $C$ columns of the factor matrices, the number of whose entries is $C\sum_{n=1}^{N}I_{n}$, the other $(K-C)$ columns do not need to be loaded into memory.
Thus, the columns of the factor matrices can be loaded by turns depending on $(k_{1},...,k_{C})$ values.
Moreover, updating $C$ columns  (lines~\ref{alg:SALS:column} through \ref{alg:SALS:entry}) can be processed by streaming the entries of $\tensor{\hat{R}}$ from disk and processing them one by one instead of loading them all at once because the entries of $\mat{B}^{(n)}_{i_{n}}$ and $\vect{c}^{(n)}_{i_{n}}$ in \eqref{eq:SALS} are the sum of the values calculated independently from each $\tensor{\hat{R}}$ entry.
Likewise, $\tensor{\hat{R}}$ computation and $\tensor{R}$ update can also be processed by streaming $\tensor{R}$ and $\tensor{\hat{R}}$, respectively.
\end{proof}

\subsection{Parallelization in Distributed Environments}
\label{sec:method:parallel}
In this section, we describe how to parallelize \CA in distributed environments such as \MR where machines do not share memory.
Algorithm \ref{alg:SALS_DIST} depicts the distributed version of \CA.

\begin{algorithm} [tbp!]
	\small
	\caption{Distributed version of \CA} \label{alg:SALS_DIST}
	\SetKwInOut{Input}{Input}
	\SetKwInOut{Output}{Output}
	\Input{
		$\tensor{X}$, $K$, $\lambda$, ${}_mS_n$ for all $m$ and $n$
	}
	\Output{
		$\mat{A}^{(n)}$ for all $n$
	}
	\vspace{1.5mm}
	distribute the ${}_m\Omega$ entries of $\tensor{X}$ to each machine $m$ \label{alg:SALS_DIST:distribute}\\
	\textbf{Parallel (P):} initialize the ${}_m\Omega$ entries of $\tensor{R}$ \\
	\textbf{P:} initialize $\mat{A}^{(n)}$ for all $n$\\
	\For{outer iter = 1..$T_{out}$}{
		\For{split iter = 1..$\lceil\frac{K}{C}\rceil$}{
			choose $k_{1}, ..., k_{C}$ {\footnotesize(from columns not updated yet)}\\
			\textbf{P:} compute ${}_m\Omega$ entries of $\tensor{\hat{R}}$\label{alg:SALS_DIST:Rhat}\\
			\For{inner iter = 1..$T_{in}$}{
				\For{n = 1..$N$}{ \label{alg:SALS_DIST:rankC}
					\textbf{P:} update {\footnotesize $\{a^{(n)}_{i_{n}k_{c}}|i_{n}\in{}_mS_{n},1\leq c \leq C\}$} using \eqref{eq:SALS} \label{alg:SALS_DIST:entry}\\
					\textbf{P:} broadcast
					{\footnotesize $\{a^{(n)}_{i_{n}k_{c}}|i_{n}\in{}_mS_{n},1\leq c \leq C\}$}
					\label{alg:SALS_DIST:broadcast}\\
				}
			}
			\textbf{P:} update the ${}_m\Omega$ entries of $\tensor{R}$ using \eqref{eq:SALS_R}\label{alg:SALS_DIST:R}\\
		}
	}
\end{algorithm}

Since update rule \eqref{eq:SALS} for each row (C parameters) of a factor matrix does not depend on the other rows in the matrix, rows in a factor matrix can be distributed across machines and updated simultaneously without affecting the correctness of \CA.
Each machine $m$ updates ${}_mS_{n}$ rows of $\mat{A}^{(n)}$ (line~\ref{alg:SALS_DIST:entry}), and for this, the $_m\Omega=\bigcup_{n=1}^{N}\left(\bigcup_{i_{n}\in {}_mS_n}\Omega^{(n)}_{i_{n}}\right)$ entries of $\tensor{X}$ are distributed to machine $m$ in the first stage of the algorithm (line~\ref{alg:SALS_DIST:distribute}).
Figure \ref{fig:CDTC} shows an example of work and data distribution in \CA.

\begin{figure}[tbp!]
    \centering
    \subfigure[Machine 1]
    {
    	\includegraphics[width=.26\linewidth]{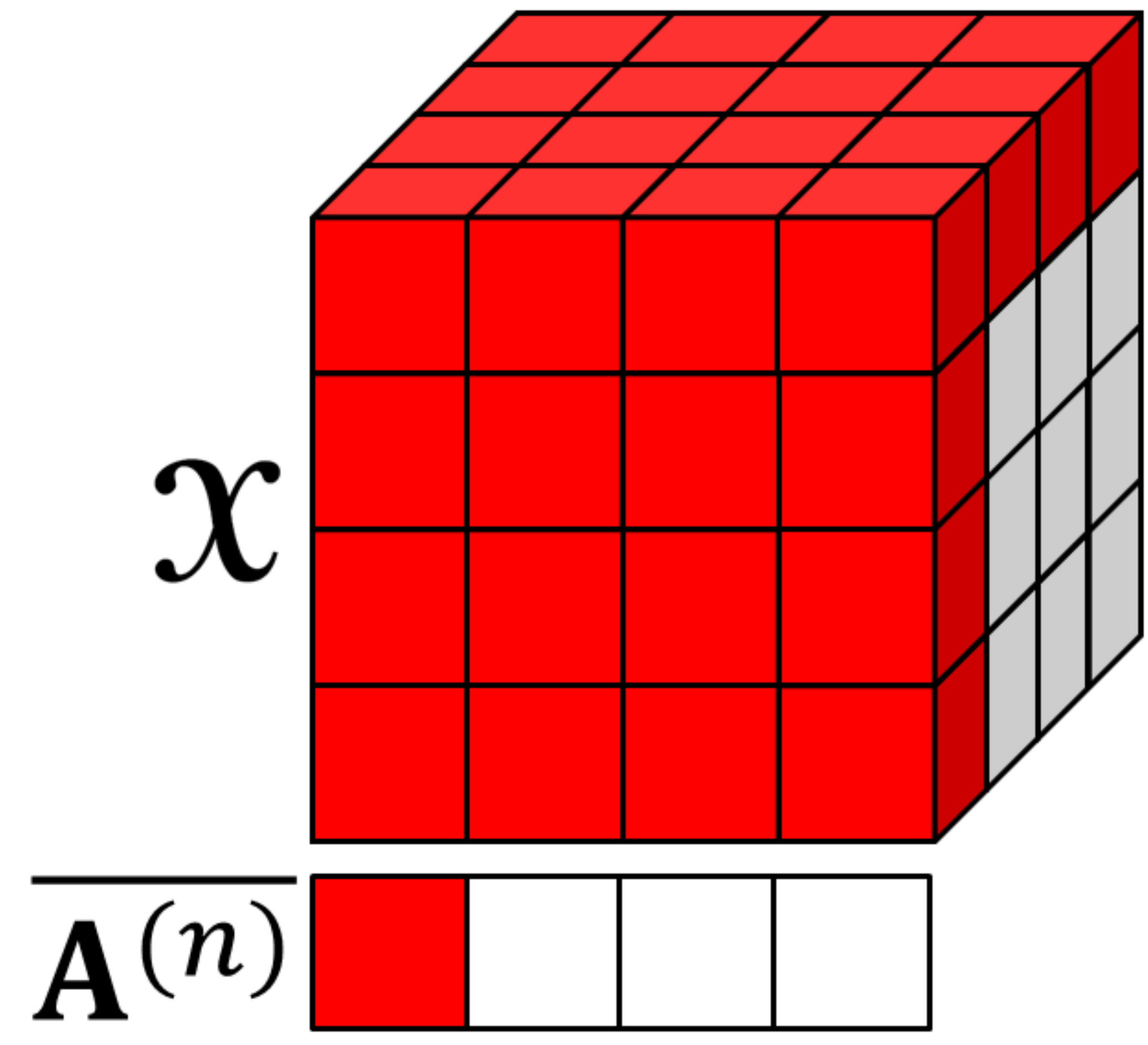}
    }
    \subfigure[Machine 2]
    {
    	\includegraphics[width=.192\linewidth]{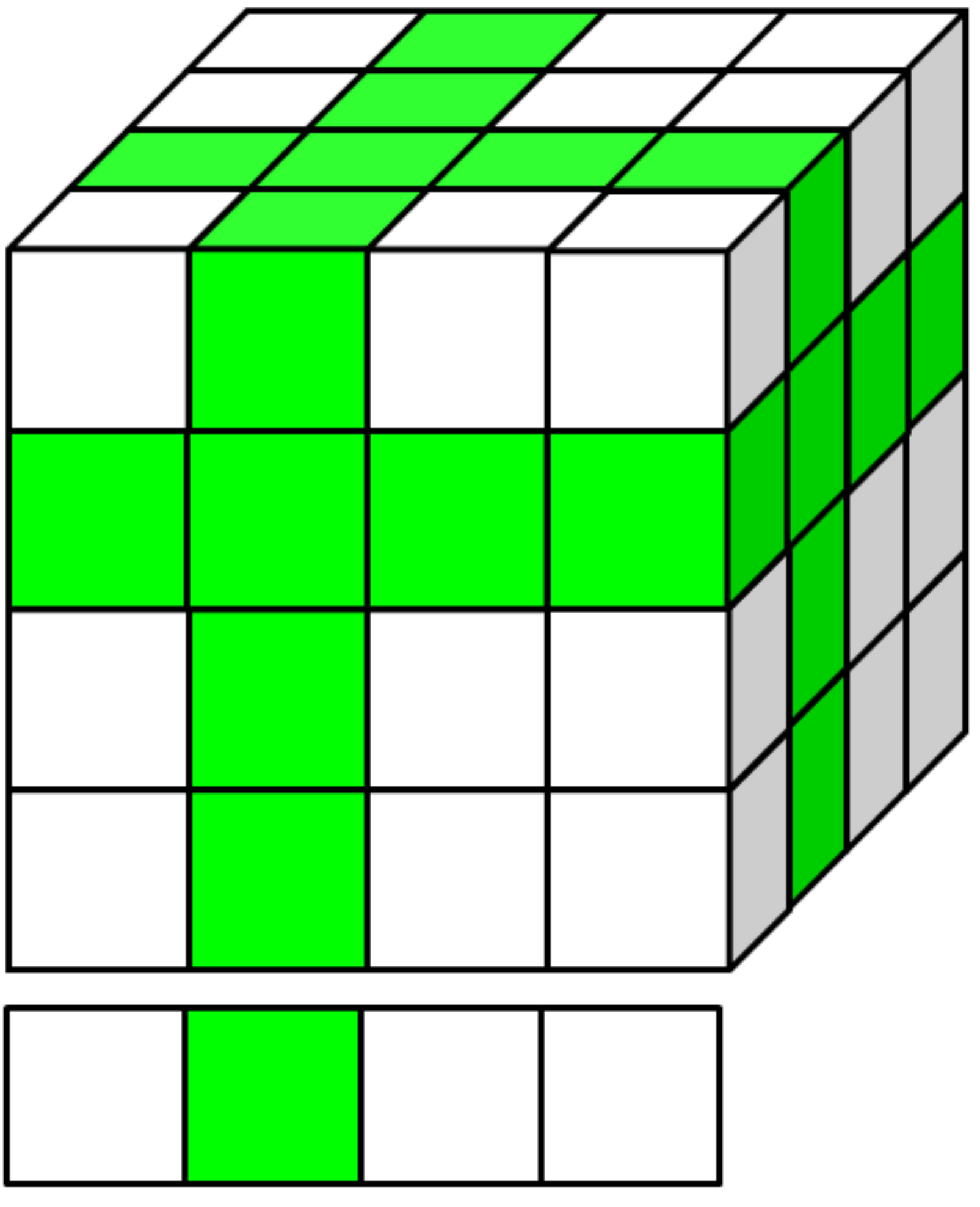}
    }
    \subfigure[Machine 3]
    {
    	\includegraphics[width=.192\linewidth]{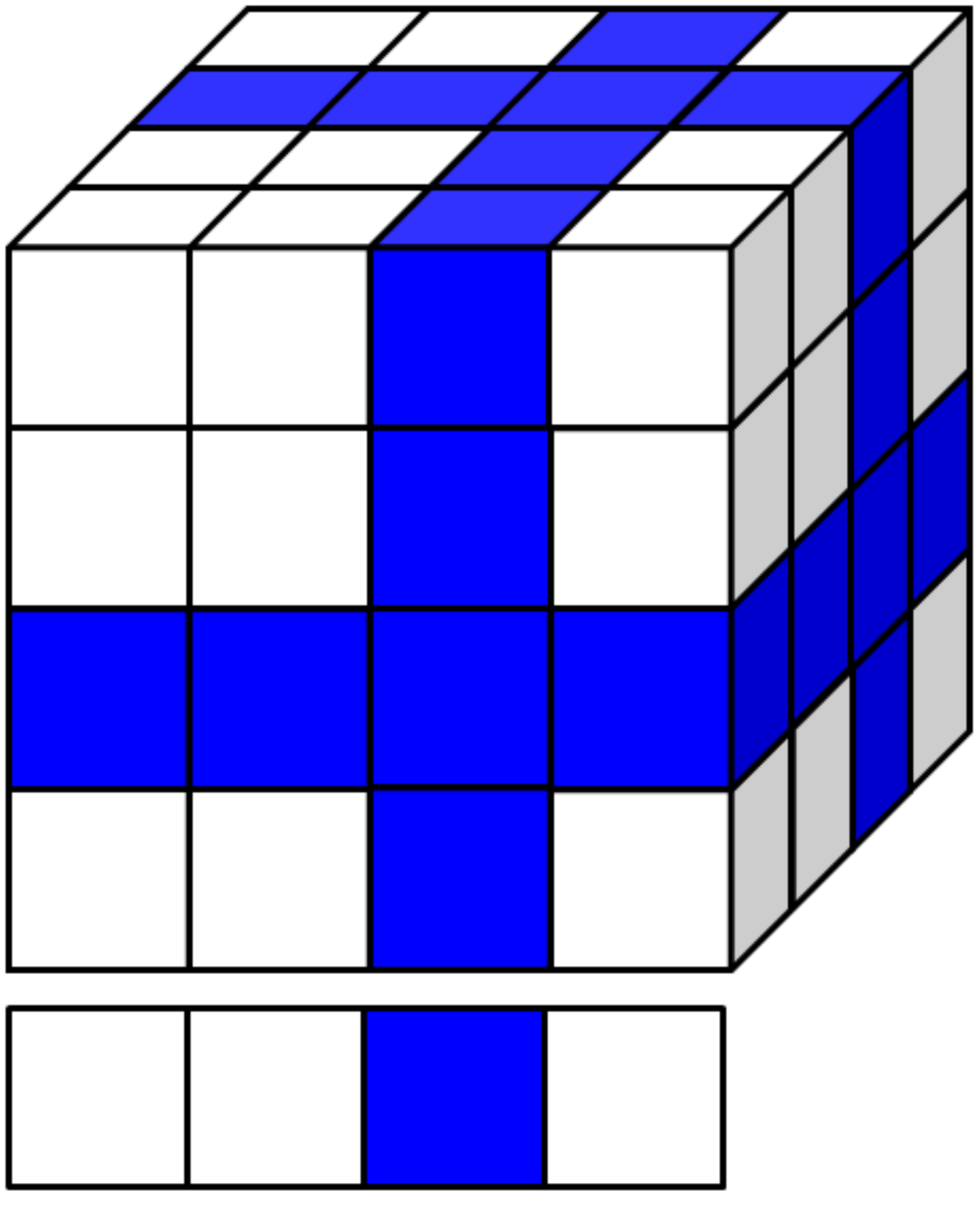}
    }
    \subfigure[Machine 4]
    {
    	\includegraphics[width=.192\linewidth]{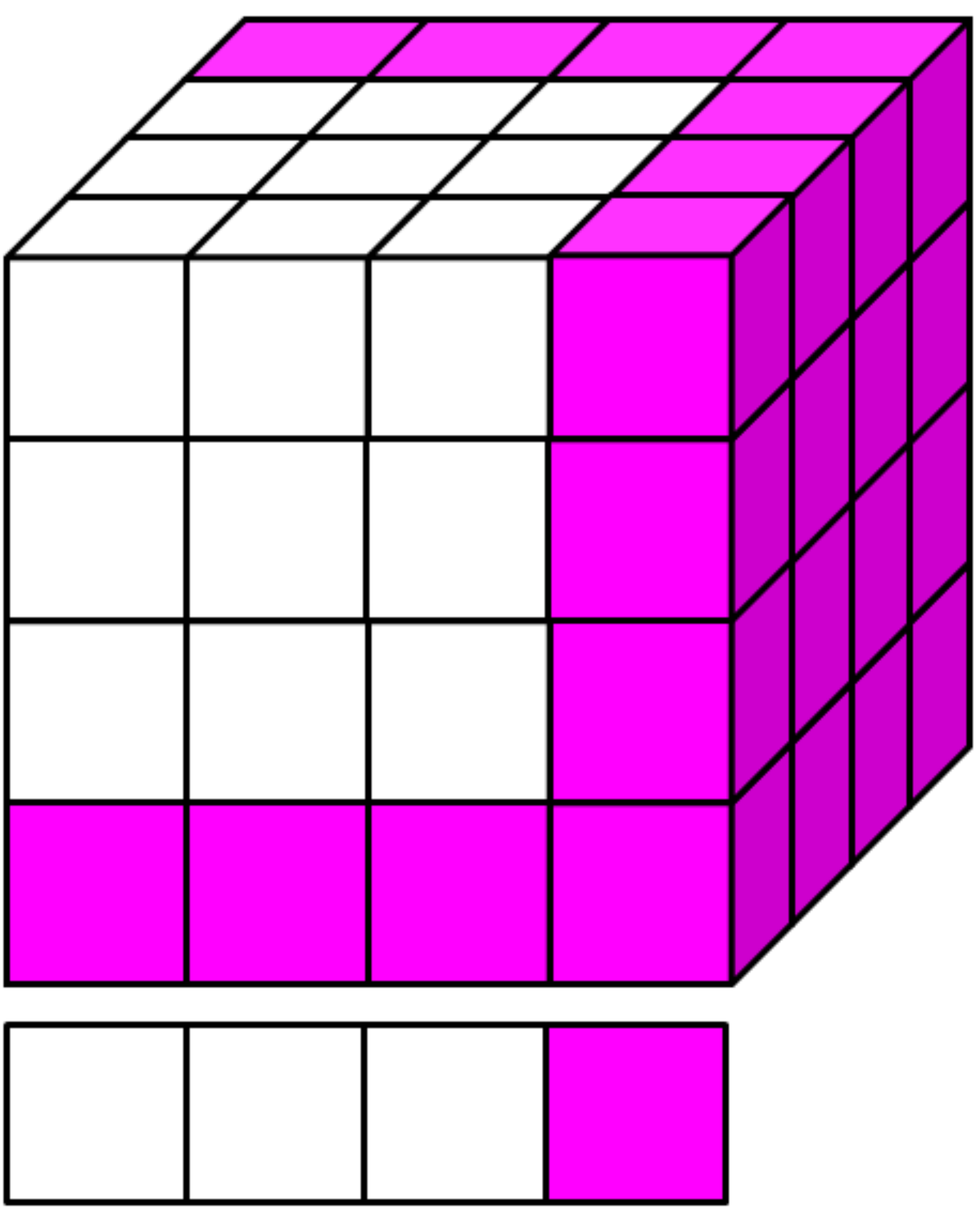}
    }
    \caption{\label{fig:CDTC} Work and data distribution of \CA in distributed environments
    when an input tensor is a three-dimensional tensor and the number of machines is four. We assume that the rows of the factor matrices are assigned to the machines sequentially.
    The colored region of $\overline{\mat{A}^{(n)}}$ (the transpose of $\mat{A}^{(n)}$) in each sub-figure corresponds to the parameters updated by each machine, resp., and that of $\tensor{X}$ corresponds to the data distributed to each machine.
    }
\end{figure}
After the update, parameters updated by each machine are broadcast to all other machines (line~\ref{alg:SALS_DIST:broadcast}).
Each machine $m$ broadcasts $C|{}_mS_{n}|$ parameters and receives $C(I_{n}-|{}_mS_{n}|)$ parameters from the other machines after each update.
The total number of parameters each machine exchanges with the other machines is $KT_{in}\sum_{n=1}^{N}I_{n}$ per outer iteration.

The running time of parallel steps in Algorithm~\ref{alg:SALS_DIST} depends on the longest running time among all machines.
Specifically, the running time of lines~\ref{alg:SALS_DIST:Rhat}, \ref{alg:SALS_DIST:entry}, and \ref{alg:SALS_DIST:R} is proportional to $\max_{m}|{}_m\Omega^{(n)}|$ where ${}_m\Omega^{(n)}=\bigcup_{i_{n}\in {}_mS_n}\Omega^{(n)}_{i_{n}}$, and that of line~\ref{alg:SALS_DIST:broadcast} is proportional to $\max_{m}|{}_mS_{n}|$.
Therefore, it is necessary to assign the rows of the factor matrices to the machines (i.e., to decide ${}_mS_{n}$) so that $|{}_m\Omega^{(n)}|$ and $|{}_mS_{n}|$ are even among all the machines.
The greedy assignment algorithm described in Algorithm~\ref{alg:index} aims to minimize $\max_{m}|{}_m\Omega^{(n)}|$ under the condition that $\max_{m}|{}_mS_{n}|$ is minimized (i.e., $|{}_mS_{n}|=I_{n}/M$ for all $n$ where $M$ is the number of machines).
For each factor matrix $\mat{A}^{(n)}$, we sort its rows in the decreasing order of $|\Omega^{(n)}_{i_n}|$ and assign the rows one by one to the machine $m$ which satisfies $|{}_mS_{n}|<\lceil I_{n}/M\rceil$ and has the smallest $|{}_m\Omega^{(n)}|$ currently.
The effects of this greedy row assignment on actual running times are analyzed in Section~\ref{sec:exp:optimization}.

\begin{algorithm} [tbp!]
\small
\caption{Greedy row assignment in \small{\CA}} \label{alg:index}
  \SetKwInOut{Input}{Input}
  \SetKwInOut{Output}{Output}
  \Input{
  	$\tensor{X}$, $M$
  }
  \Output{
    ${}_mS_n$ for all $m$ and $n$\\
  }
  \vspace{1.5mm}
  initialize $|{}_m\Omega|$ to $0$ for all $m$\\
  \For{n = 1..$N$}{
    initialize ${}_mS_{n}$ to $\emptyset$ for all $m$\\
    initialize $|{}_m\Omega^{(n)}|$ to $0$ for all $m$\\
    calculate $|\Omega^{(n)}_{i_n}|$ for all $i_n$\\
	
	\ForEach { $i_n$ {\normalfont (in decreasing order of $|\Omega^{(n)}_{i_n}|$)}}{
	   	find $m$ with $|{}_mS_{n}|<\lceil\frac{I_{n}}{M}\rceil$ and the smallest $|{}_m\Omega^{(n)}|$ \\
	   	{\footnotesize(in case of a tie, choose the machine with smaller $|{}_mS_{n}|$, and
	   	if still a tie, choose the one with smaller $|{}_m\Omega|$)}\\
	   	add $i_{n}$ to ${}_mS_{n}$ \\
	   	add $|\Omega^{(n)}_{i_n}|$ to $|{}_m\Omega^{(n)}|$ and $|{}_m\Omega|$ \\
	}
 }
\end{algorithm}

\section{Optimization on MapReduce}
\label{sec:implementation}
In this section, we describe two optimization techniques used to implement \CA and \CD on \MR, which is one of the most widely-used distributed platforms.

\begin{algorithm} [bpt!]
	\small
	\caption{Parameter update in \CA without local disk caching} \label{alg:NAIVE_CA_A}
	\SetKwInOut{Given}{Given}
	\SetKwInOut{Input}{Input}
	\SetKwInOut{Output}{Output}
	\SetKwFunction{PassMa}{Map(Key k, Value v)}
	\SetKwFunction{PassPa}{Partitioner(Key k, Value v)}
	\SetKwFunction{PassRa}{Reduce(Key k, Value v[1..r])}
	\Given{$n$, $k_{c}$ for all $c$, ${}_mS_n$ for all $m$, $\vect{a}^{(l)}_{*k_{c}}$ for all $l$ and $c$}
	\Input{$\tensor{\hat{R}}$}
	\Output{updated $a^{(n)}_{*k_{c}}$ for all $c$}
	\BlankLine
	\PassMa\\
	\Begin{
		$((i_{1},...,i_{N})$, $\hat{r}_{i_{1}...i_{N}})$ $\leftarrow$ v \\
		find $m$ where $i_{n}\in{}_mS_{n}$ \\
		emit $<(m,i_{n}),((i_{1},...,i_{N}), \hat{r}_{i_{1}...i_{N}})>$
	}
	\BlankLine
	\PassPa\\
	\Begin{
		$(m,i_{n})$ $\leftarrow$ k \\
		assign $<k,v>$ to machine $m$
	}
	\BlankLine
	\PassRa\\
	\Begin{
		$(m,i_{n})$ $\leftarrow$ k \\
		$\Omega^{(n)}_{i_{n}}$ entries of $\tensor{\hat{R}}$ $\leftarrow$ v \\ 
		update and emit $a^{(n)}_{i_{n}k_{c}}$ for all $c$
	}
\end{algorithm}

\begin{algorithm} [bpt!]
	\small
	\caption{$\tensor{R}$ update in \CA without local disk caching} \label{alg:NAIVE_CA_R}
	\SetKwInOut{Given}{Given}
	\SetKwInOut{Input}{Input}
	\SetKwInOut{Output}{Output}
	\SetKwFunction{PassMa}{Map(Key k, Value v)}
	\SetKwFunction{PassPa}{Partitioner(Key k, Value v)}
	\SetKwFunction{PassRa}{Reduce(Key k, Value v[1..r])}
	\Given{$k_{c}$ for all $c$, ${}_mS_1$ for all $m$, $\vect{a}^{(l)}_{*k_{c}}$ for all $l$ and $c$}
	\Input{$\tensor{\hat{R}}$}
	\Output{updated $\tensor{R}$}
	\BlankLine
	\PassMa\\
	\Begin{
		$((i_{1},...,i_{N})$, $\hat{r}_{i_{1}...i_{N}})$ $\leftarrow$ v \\
		find $m$ where $i_{1}\in{}_mS_{1}$ \\
		emit $<m,((i_{1},...,i_{N}), \hat{r}_{i_{1}...i_{N}})>$
	}
	\BlankLine
	\PassPa\\
	\Begin{
		$m$ $\leftarrow$ k \\
		assign $<k,v>$ to machine $m$
	}
	\BlankLine
	\PassRa\\
	\Begin{
		\ForEach{$((i_{1},...,i_{N})$, $\hat{r}_{i_{1}...i_{N}})\in$ {\normalfont v[1...r]}}{
			update $r_{i_{1}...i_{N}}$ \\
			emit $((i_{1},...,i_{N})$, $r_{i_{1}...i_{N}})$
		}
	}
\end{algorithm}

\begin{algorithm} [ht!]
	\small
	\caption{\small Data distribution in \CA with local disk caching} \label{alg:mapreduce:data}
	\SetKwInOut{Input}{Input}
	\SetKwInOut{Output}{Output}
	\SetKwFunction{PassMa}{Map(Key k, Value v)}
	\SetKwFunction{PassPa}{Partitioner(Key k, Value v)}
	\SetKwFunction{PassRa}{Reduce(Key k, Value v[1..r])}
	\Input{ $\tensor{X}$, ${}_mS_n$ for all $m$ and $n$ }
	\Output{ ${}_m\Omega^{(n)}$ entries of $\tensor{R}(=\tensor{X})$ for all $m$ and $n$ }
	\BlankLine
	\PassMa\;
	\Begin{
		$((i_{1},...,i_{N})$, $x_{i_{1}...i_{N}})$ $\leftarrow$ v \\
		\For{n = 1,...,$N$}{
			find $m$ where $i_{n}\in{}_mS_{n}$ \\
			emit $<(m,n),((i_{1},...,i_{N}), x_{i_{1}...i_{N}})>$
		}
	}
	\BlankLine
	\PassPa\;
	\Begin{
		$(m,n)$ $\leftarrow$ k \\
		assign $<k,v>$ to machine $m$
	}
	\BlankLine
	\PassRa\;
	\Begin{
		$(m,n)$ $\leftarrow$ k \\
		open a file on the local disk to cache ${}_m\Omega^{(n)}$ entries of $\tensor{R}$\\
		\ForEach{$((i_{1},...,i_{N})$, $x_{i_{1}...i_{N}})\in$ {\normalfont v[1...r]}}{
			write $((i_{1},...,i_{N})$, $x_{i_{1}...i_{N}})$ to the file\\
		}
	}
\end{algorithm}

\subsection{Local Disk Caching}
\label{sec:impl:data}
Typical \MR implementations of \CA and \CD without local disk caching run each parallel step as a separate \MR job.
Algorithms~\ref{alg:NAIVE_CA_A} and \ref{alg:NAIVE_CA_R} describe the \MR implementation of parameter update (update of $a^{(n)}_{*k_{c}}$ for all $c$) and $\tensor{R}$ update, respectively.
$\tensor{\hat{R}}$ computation can be implemented by the similar way with $\tensor{R}$ update.
in this implementation, broadcasting updated parameters is unnecessary because reducers terminate after updating their assigned parameters.
Instead, the updated parameters are saved in the distributed file system and are read at the next step (a separate job). 
Since \CA repeats both $\tensor{R}$ update and $\tensor{\hat{R}}$ computation $K/C$ times and parameter update $KT_{in}N/C$ times at every outer iteration, this implementation repeats distributing $\tensor{R}$ or $\tensor{\hat{R}}$ across machines (the mapper stage of Algorithms~\ref{alg:NAIVE_CA_A} and \ref{alg:NAIVE_CA_R}) $T_{out}K(T_{in}N+2)/C$ times, which is inefficient.

Our implementation reduces this inefficiency by caching data to local disk once they are distributed.
In the \CA implementation with local disk caching, $\tensor{X}$ entries are distributed across machines and cached in the local disk during the map and reduce stages (Algorithm~\ref{alg:mapreduce:data}); and the rest part of \CA runs in the close stage (cleanup stage in Hadoop) using cached data.
Our implementation streams the cached data from the local disk instead of distributing entire $\tensor{R}$ or $\tensor{\hat{R}}$ from the distributed file system when updating factor matrices. For example, ${}_m\Omega^{(n)}$ entries of $\tensor{\hat{R}}$ are streamed from the local disk when the columns of $\mat{A}^{(n)}$ are updated.
The effect of this local disk caching on the actual running time is analyzed in Section~\ref{sec:exp:optimization}.

\subsection{Direct Communication}
\label{sec:impl:communication}
In \MR, it is generally assumed that reducers run independently and do not communicate directly with each other.
However, we adapt the direct communication method using the distributed file system in \cite{beutel2014flexifact} to broadcast parameters among reducers efficiently.
The implementation of parameter broadcast in \CA (i.e., broadcast of $\vect{a}^{(n)}_{*k_{c}}$ for all $c$) based on this method is described in Algorithm~{\ref{alg:mapreduce:communication}} where ${}_{m}\vect{a}^{(n)}_{*k_{c}}$ denotes $\{a^{(n)}_{i_{n}k_{c}}|i_{n}\in {}_{m}S_{n}\}$.

\begin{algorithm} [t!]
	\small
	\caption{Parameter broadcast in \CA} \label{alg:mapreduce:communication}
	\SetKwInOut{Input}{Input}
	\SetKwInOut{Output}{Output}
	\Input{${}_{m}\vect{a}^{(n)}_{*k_{c}}$ for all $c$ (parameters to broadcast)}
	\Output{ $\vect{a}^{(n)}_{*k_{c}}$ for all c (parameters received from others)}
	\Begin{
		create a data file ${}_mA$ on the distributed file system (DFS) \\
		write ${}_m\vect{a}^{(n)}_{*k}$ on the data file \\
		create a dummy file ${}_mD$ on DFS \\
		\While{not all data files were read}{
			get the list of dummy files from DFS \\
			\ForEach {${}_{m'}D$ {\normalfont in the list}}{
				\If{${}_{m'}A$ were not read}{
					read ${}_{m'}\vect{a}^{(n)}_{*k}$ from ${}_{m'}A$ \\
				}
			}
		}
		
	}
\end{algorithm}
%

\section{Experiments}
\label{sec:experiment}
To evaluate \CA and \CD, we design and conduct experiments to answer the following questions:
\begin{itemize*}
  \item \textbf{Q1: Data scalability (Section~\ref{sec:exp:scale:data}).} How do \CA, \CD, and other methods scale with regard to the following properties of an input tensor: dimension, the number of observations, mode length, and rank?
  \item \textbf{Q2: Machine scalability (Section~\ref{sec:exp:scale:machine}).} How do \CA, \CD, and other methods scale with regard to the number of machines?
  \item \textbf{Q3: Convergence (Section~\ref{sec:exp:converge}).} How quickly and accurately do \CA, \CD, and other methods factorize real-world tensors?
  \item \textbf{Q4: Optimization (Section~\ref{sec:exp:optimization}).} How much do the local disk caching and the greedy row assignment improve the speed of \CA and \CD? Can these optimization techniques be applied to other methods?
  \item \textbf{Q5: Effects of $\mathbf{T_{in}}$ (Section~\ref{sec:exp:T})} How do different numbers of inner iterations ($T_{in}$) affect the convergence of \CD? 
  \item \textbf{Q6: Effects of $\mathbf{C}$ (Section~\ref{sec:exp:C})} How do different numbers of columns updated at a time ($C$) affect the convergence of \CA?
\end{itemize*}
Other methods include ALS, \FF, and PSGD, which are explained in Section~\ref{sec:prelim:tc_alg}.
All experiments are focused on the distributed version of each method, which is the most suitable to achieve our purpose of handling large-scale data.

\begin{figure*}[tbp!]
	\vspace{-2mm}
    \subfigure[Dimension]{
    	\includegraphics[width=0.23\linewidth] {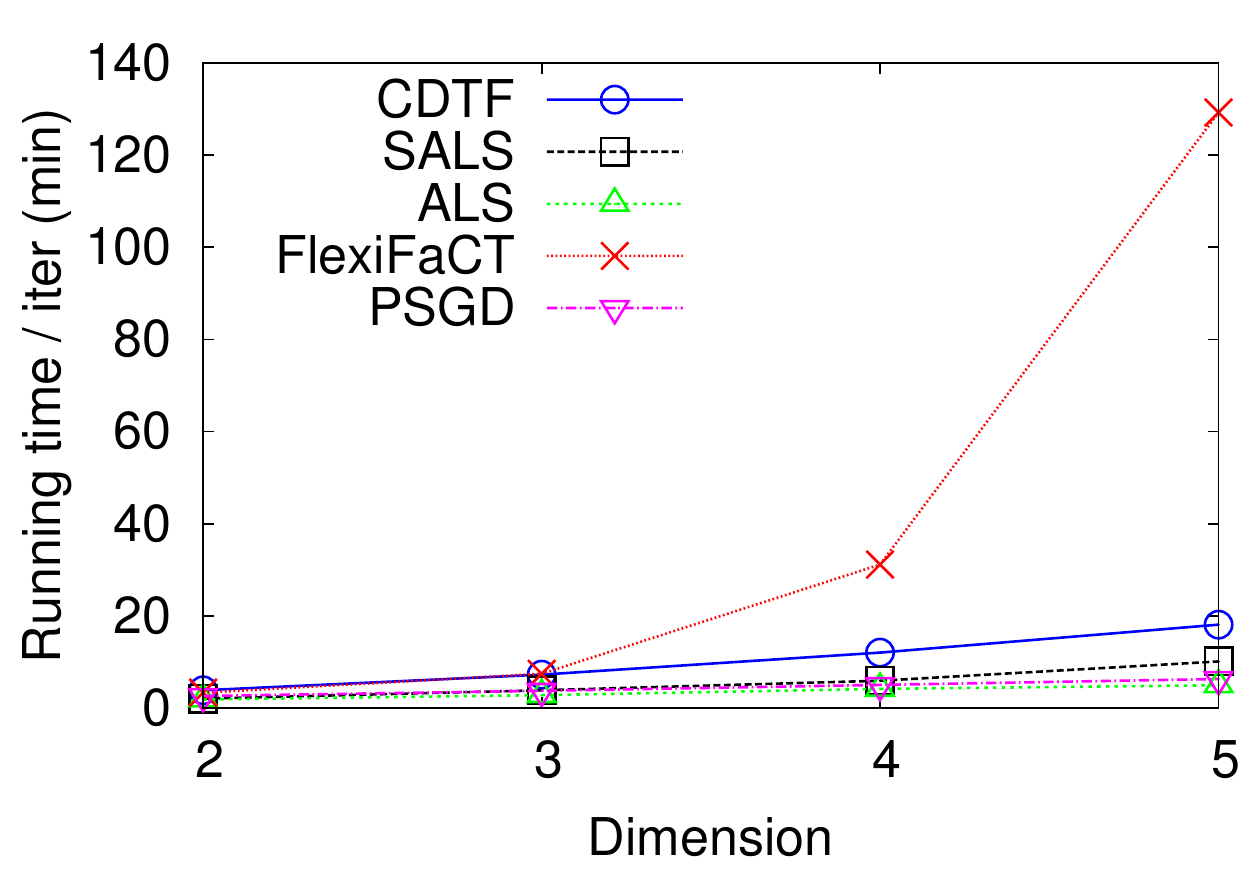}
    	\label{fig:data_scale:dimension}
    }
    \subfigure[Number of observations]{
    	\includegraphics[width=0.23\linewidth] {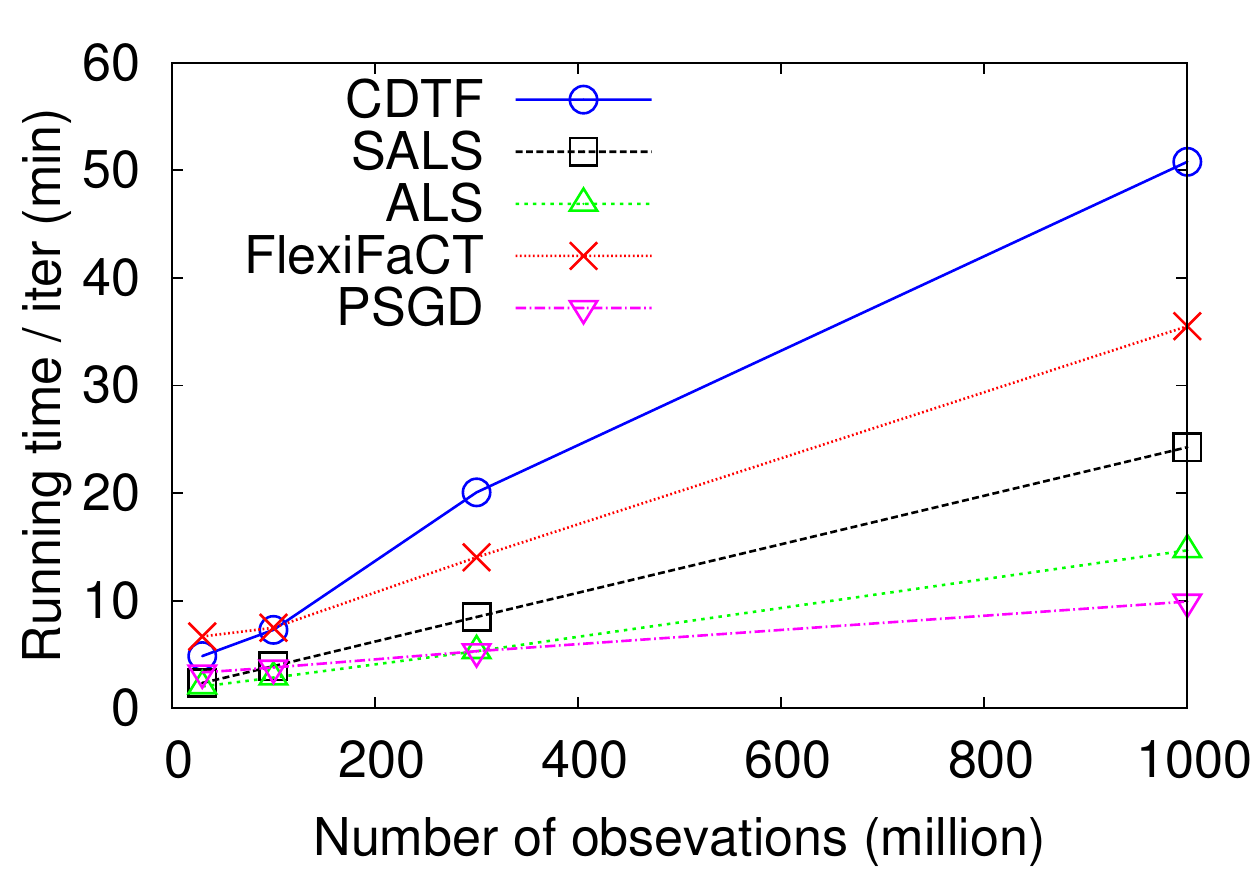}
    	\label{fig:data_scale:omega}
    }
    \subfigure[Mode length]{
    	\includegraphics[width=0.23\linewidth] {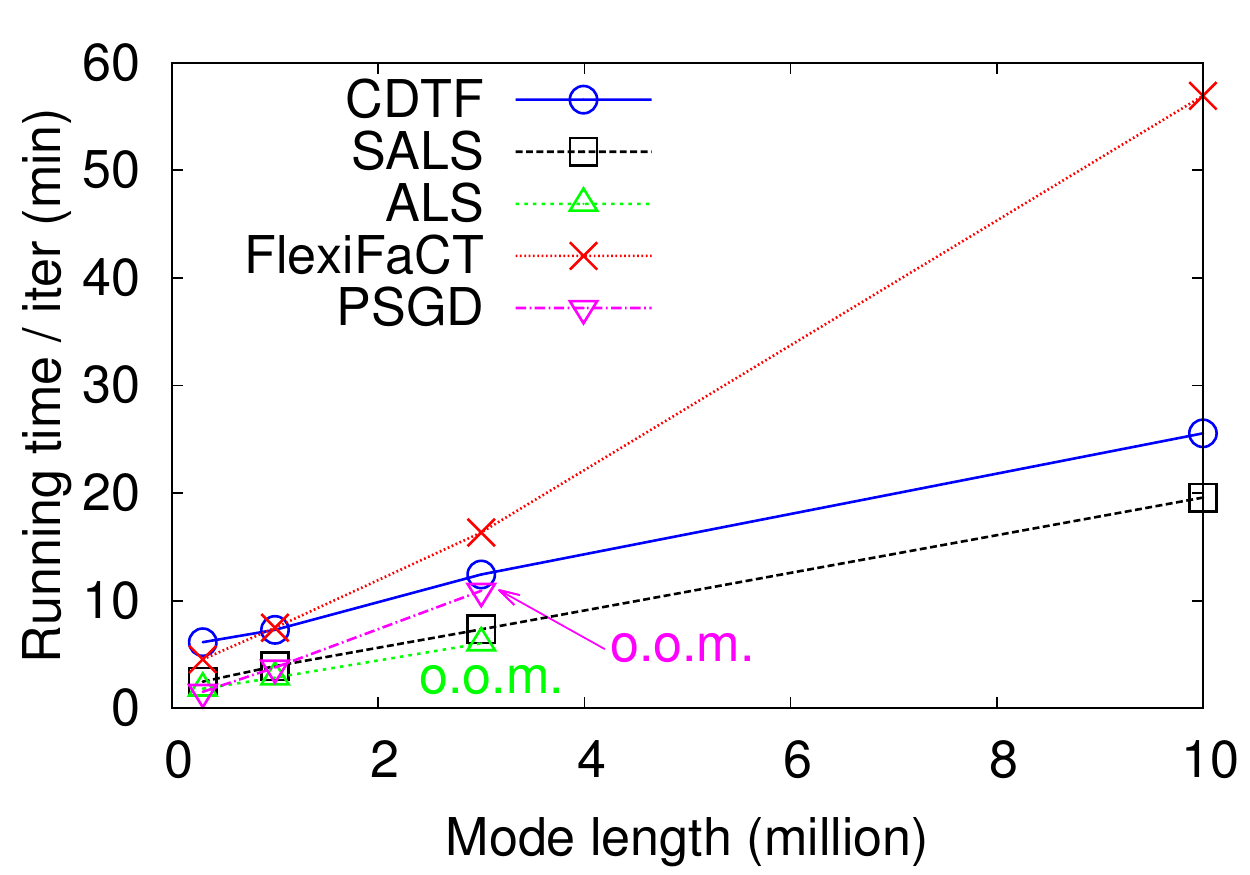}
    	\label{fig:data_scale:mode}
    }
    \subfigure[Rank]{
	    \includegraphics[width=0.23\linewidth] {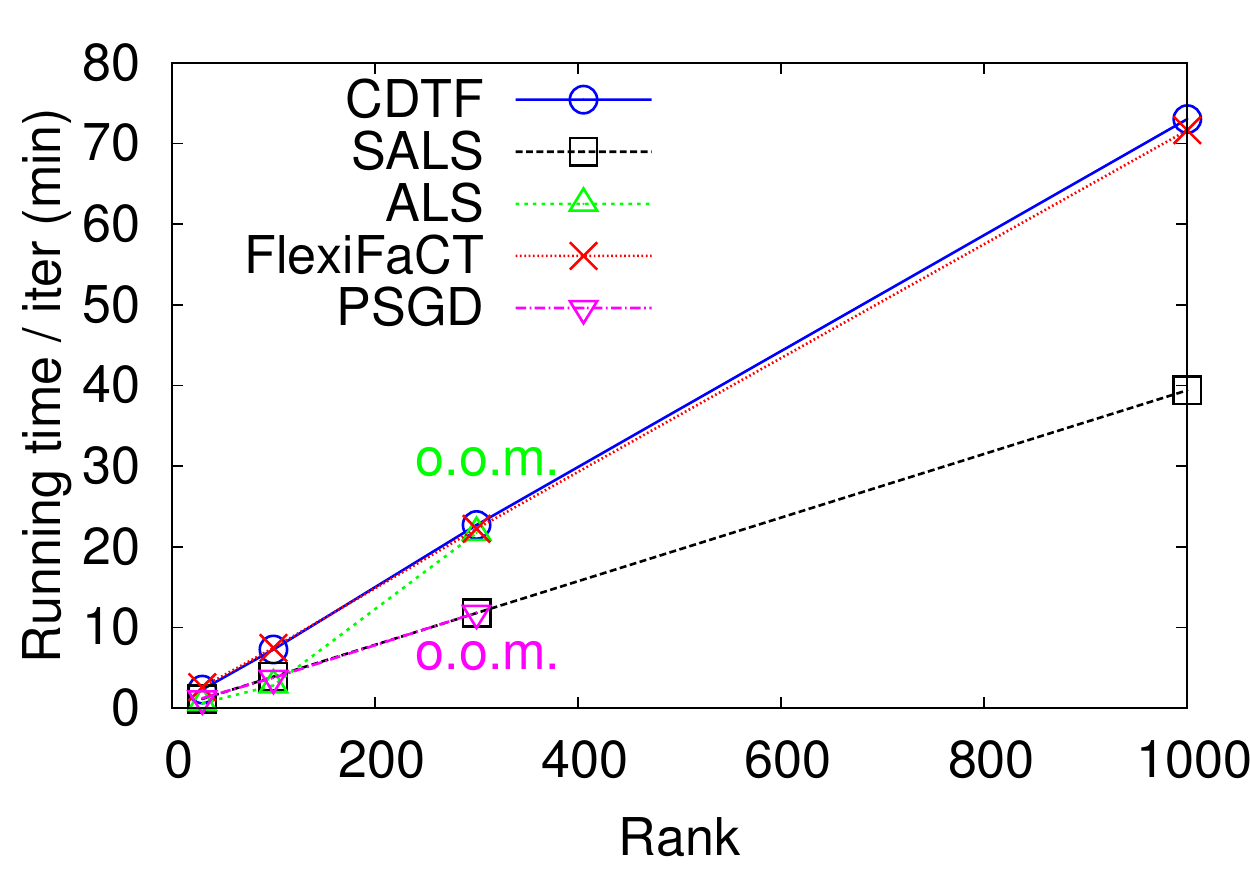}
    	\label{fig:data_scale:rank}
    }
    \caption{\label{fig:data_scale}
        Scalability with regard to each aspect of data. o.o.m. : out of memory. Only \CA and \CD scale with all the aspects.
   }
\end{figure*}

\begin{table}[tbp]
	\centering
	\caption{Summary of real-world datasets.}
	\begin{tabular}{c|ccc}
		\toprule
		\textbf{} & \textbf{Movielens}${}_4$ & \textbf{Netflix}${}_3$ & \textbf{Yahoo-music}${}_4$\\
		\midrule
		$N$ & 4 & 3 & 4 \\
		$I_{1}$ & 715,670 & 2,649,429 & 1,000,990 \\
		$I_{2}$ & 65,133 & 17,770 & 624,961 \\
		$I_{3}$ & 169 & 74 & 133 \\
		$I_{4}$ & 24 & - & 24 \\
		$|\Omega|$ & 93,012,740 & 99,072,112 & 252,800,275 \\
		$|\Omega|_{test}$ & 6,987,800 & 1,408,395 & 4,003,960 \\
		$K$ & 20 & 40 & 80 \\
		$\lambda$ & 0.01 & 0.02 & 1.0 \\
		$\eta_{0}$ & 0.01 & 0.01 & $10^{-5}$ (\FF)\\
		$ $ & & & $10^{-4}$ (PSGD) \\
		\bottomrule
	\end{tabular}
	\label{tab:real_data}
\end{table}

\begin{table}[tbp]
	\centering
	\caption{Scale of synthetic datasets. B: billion, M: million, K: thousand. The length of every mode is equal to $I$.}
	\begin{tabular}{c|cccc}
		\toprule
		\textbf{} & \textbf{S1} & \textbf{S2} (default) & \textbf{S3} & \textbf{S4}\\
		\midrule
		$N$ & 2 & 3 & 4 & 5 \\
		$I$ & 300K & 1M & 3M & 10M \\
		$|\Omega|$ & 30M & 100M & 300M & 1B \\
		$K$ & 30 & 100 & 300 & 1K \\
		\bottomrule
	\end{tabular}
	\label{tab:syn_data}
\end{table}

\subsection{Experimental Settings}
\label{sec:exp:setting}
\subsubsection{Cluster}
We run experiments on a 40-node Hadoop cluster.
Each node has an Intel Xeon E5620 2.4GHz CPU.
The maximum heap memory size per reducer is set to 8GB.

\subsubsection{Data}
We use both real-world and synthetic datasets most of which are available at \url{http://kdmlab.org/sals}. 
The real-world tensor data used in our experiments are summarized in Table~\ref{tab:real_data} with the following details:
\begin{itemize*}
	\item \textbf{Movielens$_{4}$\footnote{http://grouplens.org/datasets/movielens}:}
	Movie rating data from MovieLens, an online movie recommender service.
	We convert them into a four-dimensional tensor where the third mode and the fourth mode correspond to (year, month) and hour-of-day when the movie is rated, respectively.
	The rates range from 1 to 5.
	\item \textbf{Netflix$_{3}$\footnote{http://www.netflixprize.com}:}
	Movie rating data used in Netflix prize.
	We regard them as a three-dimensional tensor where the third mode corresponds to (year, month) when the movie is rated.
	The rates range from 1 to 5.
	\item \textbf{Yahoo-music$_{4}$\footnote{http://webscope.sandbox.yahoo.com/catalog.php?datatype=c}:}
	Music rating data used in KDD CUP 2011.
	We convert them into a four-dimensional tensor by the same way in Movielens$_{4}$.
	Since exact year and month are not provided, we use the values obtained by dividing the provided data (the number of days passed from an unrevealed date) by 30.
	The rates range from 0 to 100.
\end{itemize*}
For reproducibility, we use the original training/test split offered by the data providers. 
Synthetic tensors are created by the procedure used in \cite{niu2011hogwild} to create Jumbo dataset.
The scales of the synthetic datasets used are summarized in Table~\ref{tab:syn_data}.

\subsubsection{Implementation and Parameter Settings}
All the methods in Table~\ref{tab:algorithms} are implemented in Java with Hadoop 1.0.3.
The local disk caching, the direct communication, and the greedy row assignment are applied to all the methods if possible.
All our implementations use weighted-$\lambda$-regularization \cite{zhou2008large}.
For \CA and \CD, $T_{in}$ is set to 1, and $C$ is set to 10, unless otherwise stated.
The learning rate of \FF and PSGD at $t$th iteration is set to $2\eta_{0}/(1+t)$ following the open-sourced \FF implementation (\url{http://alexbeutel.com/l/flexifact/}).
The number of reducers is set to 5 for \FF, 20 for PSGD, and 40 for the other methods, each of which leads to the best performance on the machine scalability test in Section~\ref{sec:exp:scale:machine}, unless otherwise stated.

\subsection{Data Scalability}
\label{sec:exp:scale:data}

\subsubsection{Scalability with Each Factor (Figure~\ref{fig:data_scale})}
We measure the scalability of \CD, \CA, and the competitors with regard to the dimension, number of observations, mode length, and rank of an input tensor.
When measuring the scalability with regard to a factor, the factor is scaled up from S1 to S4 while all other factors are fixed at S2 as summarized in Table~\ref{tab:syn_data}.
As seen in Figure~\ref{fig:data_scale:dimension}, \FF does not scale with dimension because of its communication cost, which increases exponentially with dimension.
ALS and PSGD are not scalable with mode length and rank due to their high memory requirements as Figures~\ref{fig:data_scale:mode} and \ref{fig:data_scale:rank} show. They require up to 11.2GB, which is 48$\times$ of 234MB that \CD requires and 10$\times$ of 1,147MB that \CA requires.
Moreover, the running time of ALS increases rapidly with rank owing to its cubically increasing computational cost.
Only \CA and \CD are scalable with all the factors as summarized in Table~\ref{tab:scale}.
Their running times increase linearly with all the factors except dimension, with which they increase slightly faster due to the quadratically increasing computational cost.

\begin{figure}[tbp]
    \hspace{-0.04\linewidth}
	\subfigure[Overall scalability]
    {
		\includegraphics[width=0.48\linewidth]{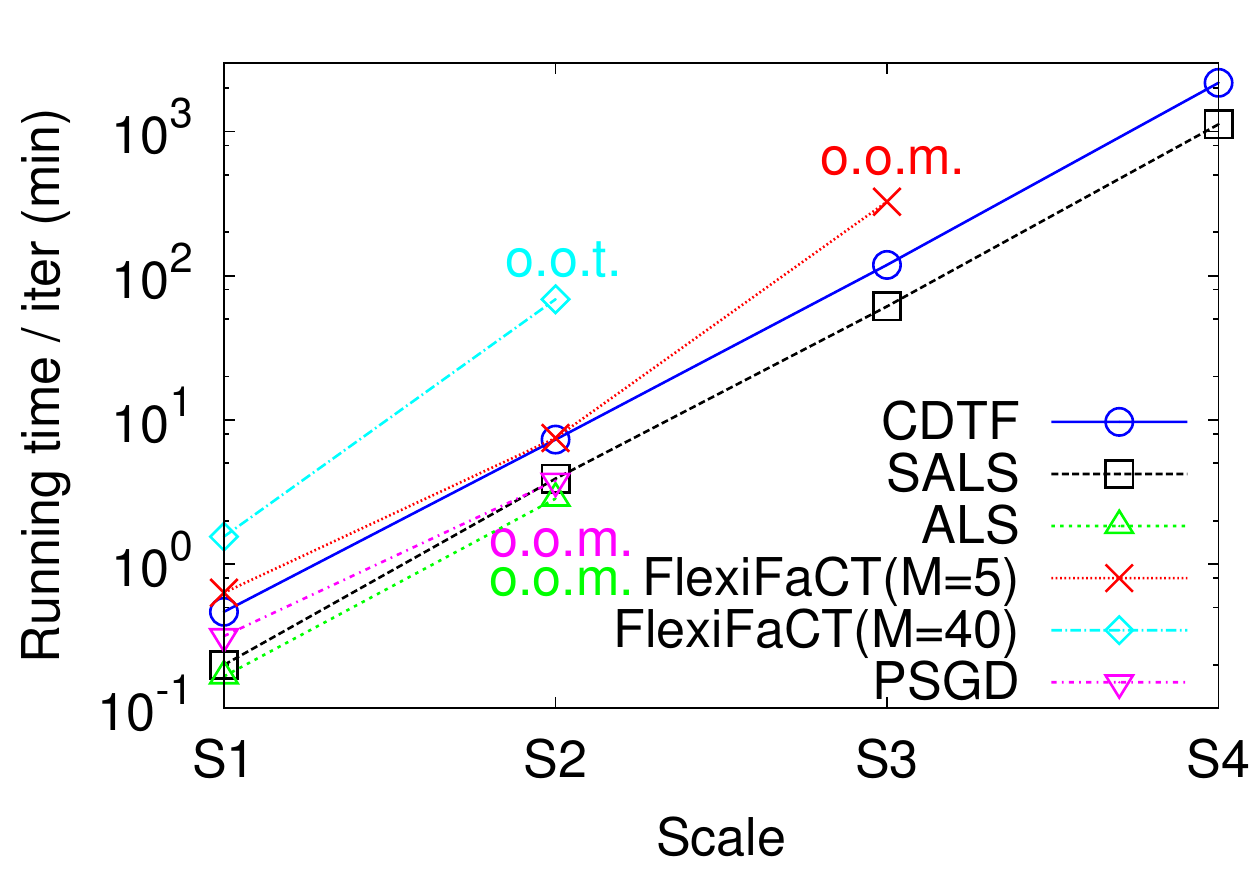}
		\label{fig:data_scale_overall}
	}
	\subfigure[Machine scalability]
	{
		\includegraphics[width=.48\linewidth]{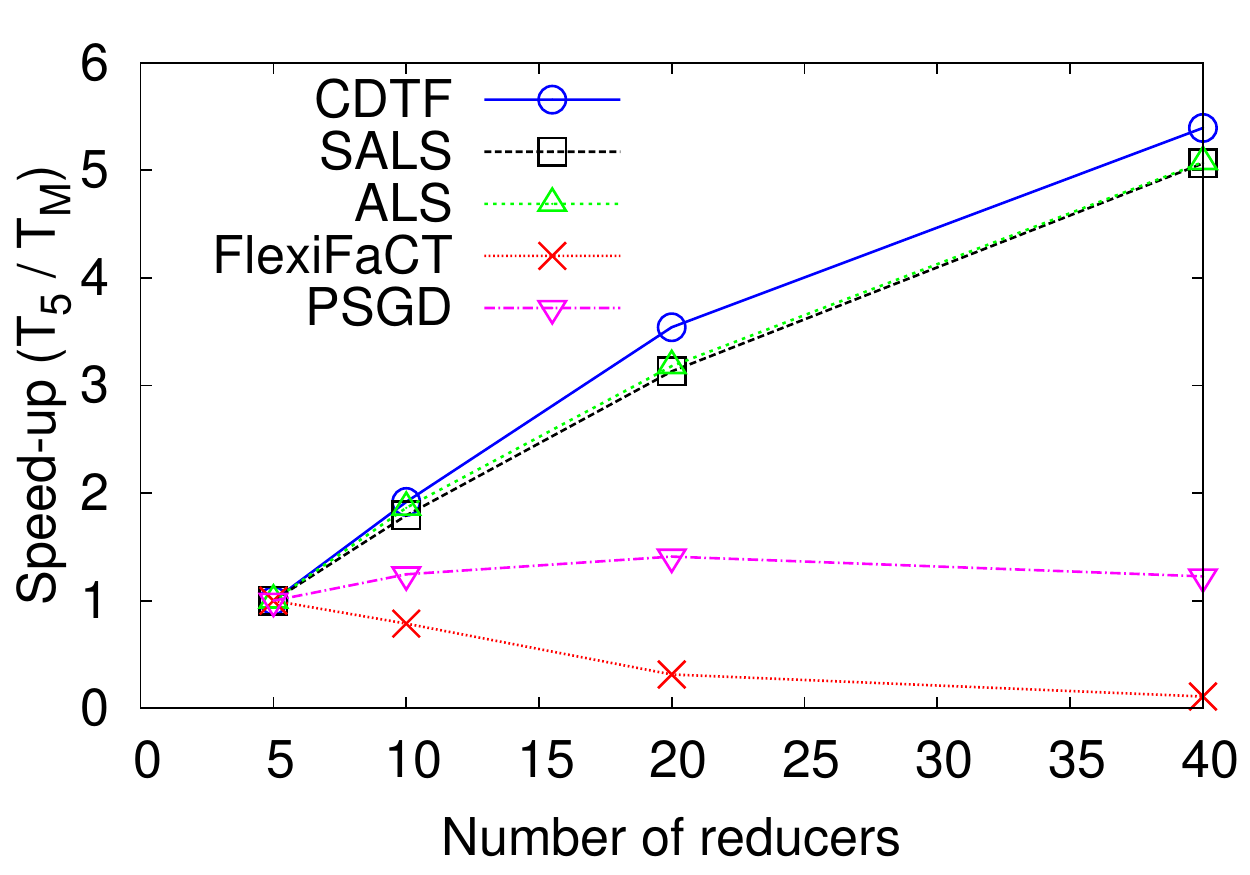}
		\label{fig:machine}
	}
    \caption{
    (a) Overall scalability. o.o.m. : out of memory, o.o.t. : out of time (takes more than a week).
    Only \CA and \CD scale up to the largest scale S4.
    (b) Machine scalability. Computations of \CA and \CD are efficiently distributed across machines.
    }
\end{figure}

\subsubsection{Overall Scalability (Figure~\ref{fig:data_scale_overall})}
We measure the scalability of the methods by scaling up all the factors together from S1 to S4.
The scalability of \FF with five machines, ALS, and PSGD is limited owing to their high memory requirements.
ALS and PSGD require almost 186GB to handle S4, which is 493$\times$ of 387MB that \CD requires and 100$\times$ of 1,912MB that \CA requires.
\FF with 40 machines does not scale over S2 due to its rapidly increasing communication cost.
Only \CA and \CD scale up to S4, and there is a trade-off between them: \CA runs faster, and \CD is more memory-efficient.

\subsection{Machine Scalability (Figure~\ref{fig:machine})}
\label{sec:exp:scale:machine}
We measure the speed-ups ($T_{5}/T_{M}$ where $T_{M}$ is the running time with $M$ reducers) of the methods on the S2 scale dataset by increasing the number of reducers.
The speed-ups of \CD, \CA, and ALS increase linearly at the beginning and then flatten out slowly owing to their fixed communication cost which does not depend on the number of reducers.
The speed-up of PSGD flattens out fast, and PSGD even slightly slows down at $40$ reducers because of increased overhead.
\FF slows down as the number of reducers increases because of its rapidly increasing communication cost.

\subsection{Convergence (Figure~\ref{fig:conv})}
\label{sec:exp:converge}
We compare how quickly and accurately each method factorizes real-world tensors.
Accuracies are calculated at each iteration by root mean square error (RMSE) on a held-out test set, which is a measure commonly used by recommender systems.
Table~\ref{tab:real_data} describes $K$, $\lambda$, and $\eta_{0}$ values used for each dataset.
They are determined by cross validation.
Owing to the non-convexity of \eqref{eq:TF}, each algorithm may converge to local minima with different accuracies.
In all datasets (results on the Movielens$_{4}$ dataset are omitted for space reasons), \CA is comparable with ALS, which converges the fastest to the best solution, and \CD follows them.
PSGD converges the slowest to the worst solution due to the non-identifiability of \eqref{eq:TF} \cite{gemulla2011large}.
\begin{figure}[tbp]
    \hspace{-0.04\linewidth}
	\subfigure[Netflix$_{3}$]
	{
		\includegraphics[width=0.48\linewidth]{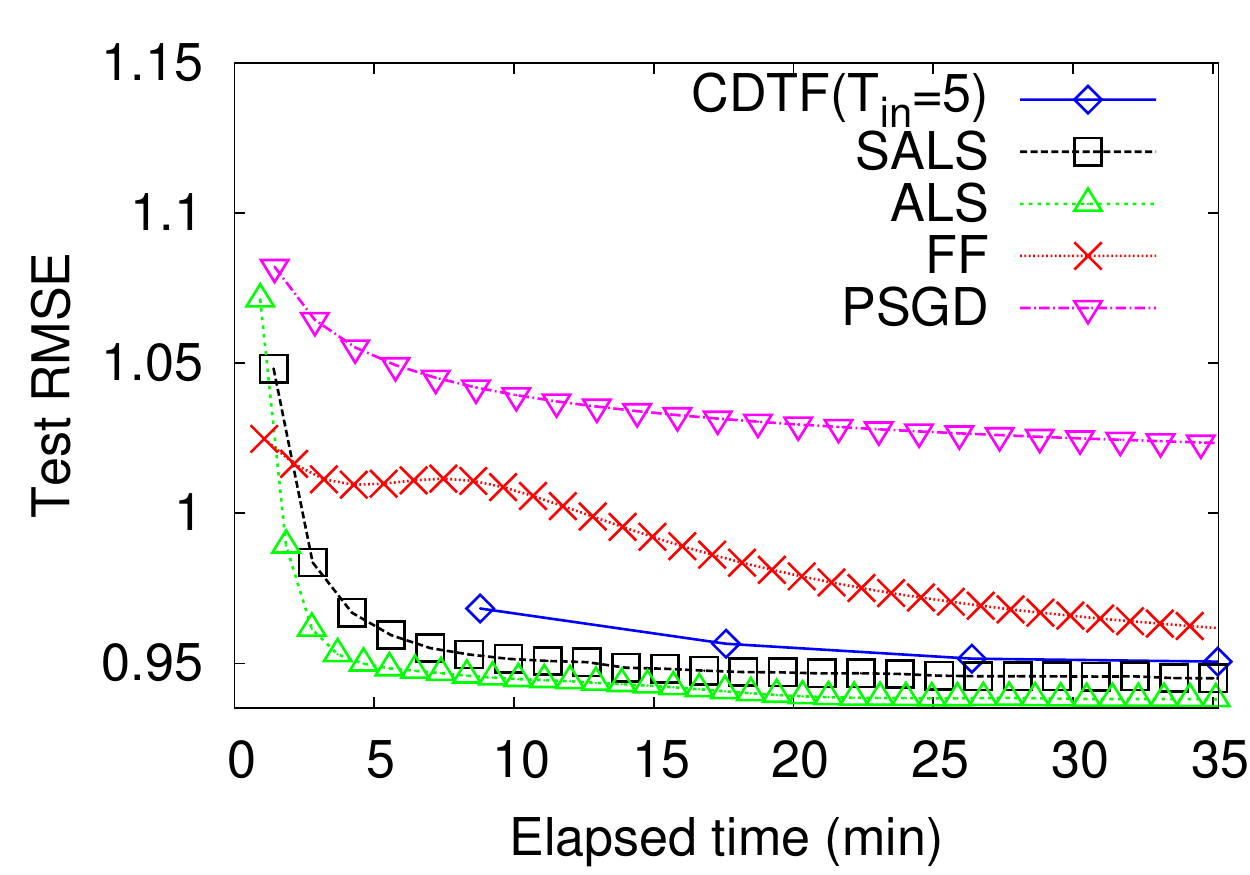}
		\label{fig:conv:nf:hadoop}
	}
	\subfigure[Yahoo-music${}_4$]
	{
		\includegraphics[width=0.48\linewidth]{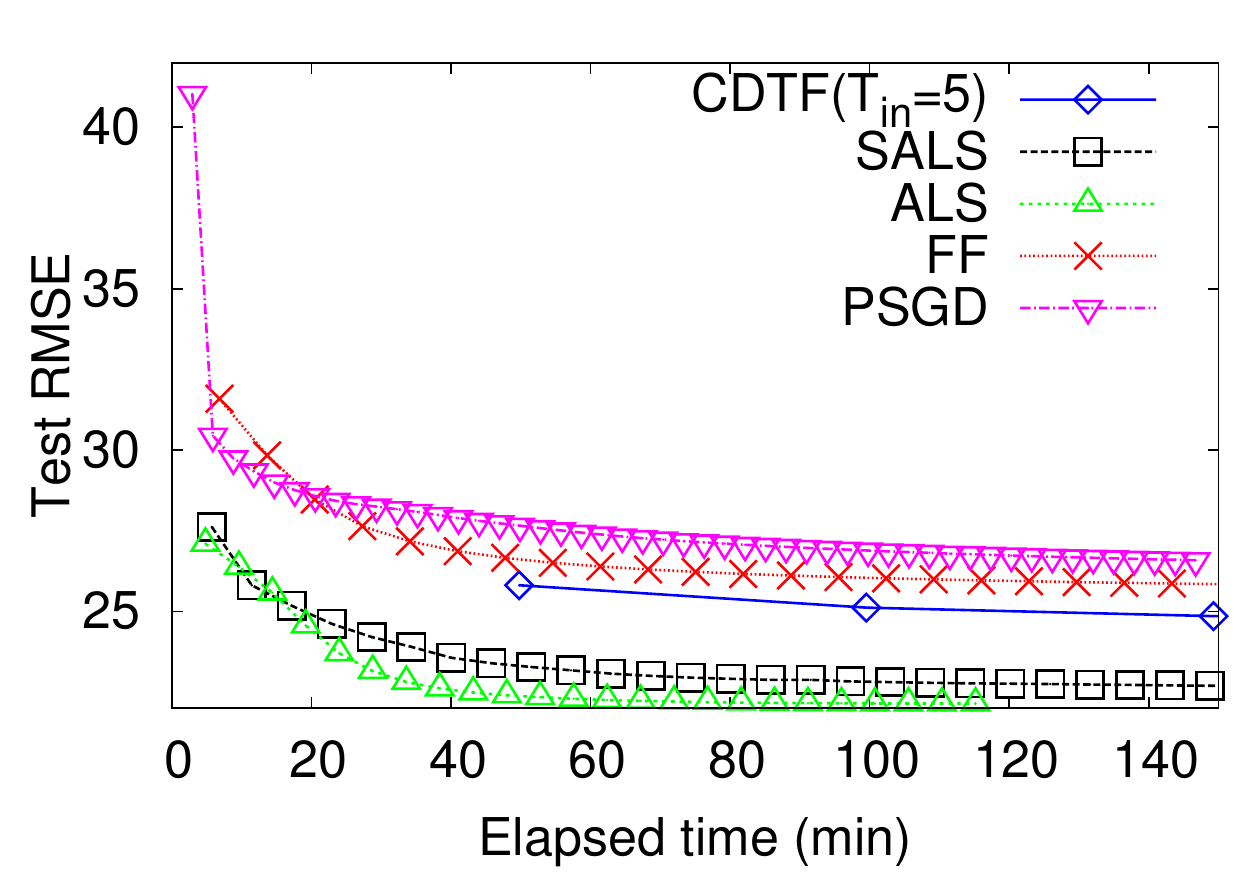}
		\label{fig:conv:yahoo:hadoop}
	}
	\caption{\label{fig:conv}
    Convergence speed on real-world datasets.
    \CA is comparable with ALS, which converges fastest to the best solution, and \CD follows them.
    }
\end{figure}

\subsection{Optimization (Figure~\ref{fig:impl})}
\label{sec:exp:optimization}
We measure how our proposed optimization techniques, the local disk caching and the greedy row assignment, affect the running time of \CD, \CA, and the competitors on real-world datasets.
The direct communication method explained in Section~\ref{sec:impl:communication} is applied to all the implementations if necessary.
The local disk caching speeds up \CD up to $65.7\times$, \CA up to $15.5\times$, and the competitors up to $4.8\times$.
The speed-ups of \CA and \CD are the most significant because of the highly iterative nature of \CA and \CD.
Additionally, the greedy row assignment speeds up \CD up to $1.5\times$; \CA up to $1.3\times$; and the competitors up to $1.2\times$ compared with the second best one.
It is not applicable to PSGD, which does not distribute parameters row by row.

\begin{figure}[tbp!]
\centering
	\hspace{-0.04\linewidth}
    \subfigure[Netflix${}_3$]{
       \includegraphics[width=0.48\linewidth] {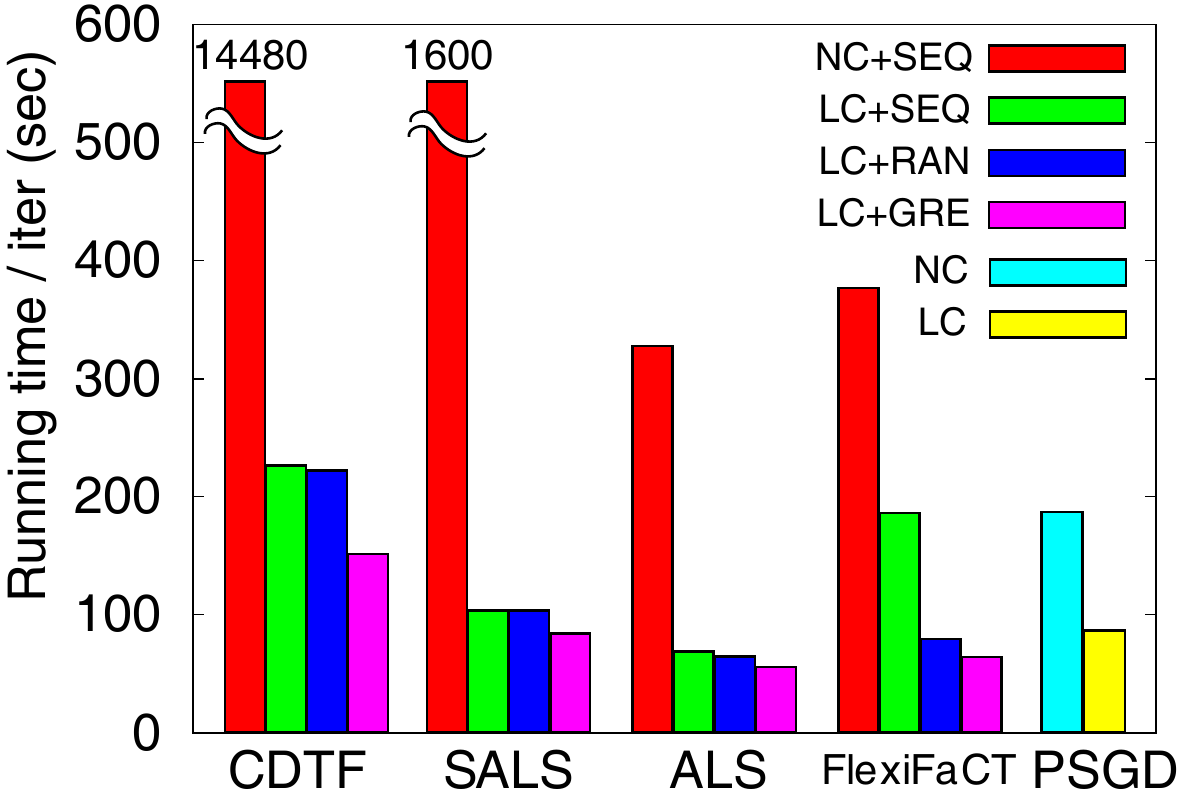}
     }
     \hspace{-0.045\linewidth}
     \subfigure[Yahoo-music${}_4$]{
       \includegraphics[width=0.48\linewidth] {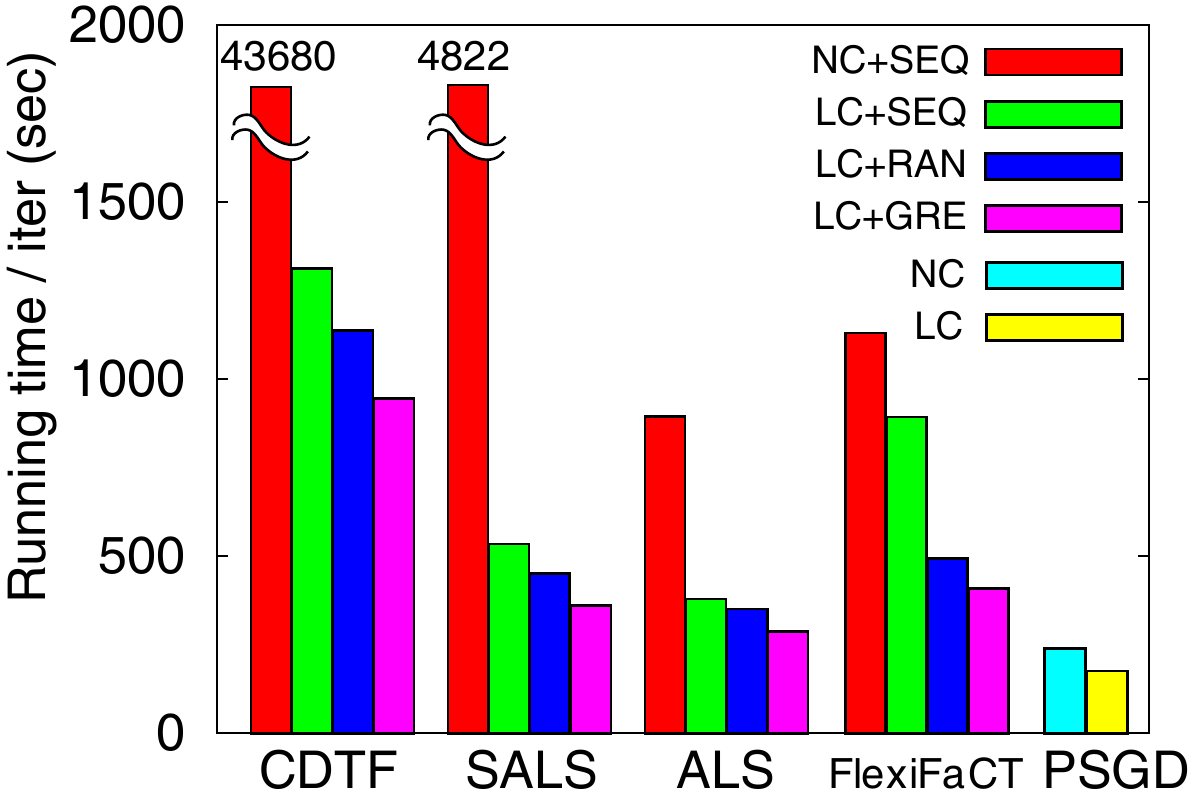}
     }
    \caption{Effects of optimization techniques on running times. NC: no caching, LC: local disk caching, SEQ: sequential row assignment\protect\footnotemark, RAN: random row assignment, GRE: greedy row assignment. Our proposed optimization techniques (LC+GRE) significantly accelerate \CD, \CA, and also their competitors.}
    \label{fig:impl}
\end{figure}
\footnotetext{ ${}_mS_{n}=\{i_{n}\in\mathbb{N}|\frac{I_{n}\times (m-1)}{M}< i_{n}\leq \frac{I_{n}\times m}{M}\}$} 

\subsection{Effects of Inner Iterations (Figure~\ref{fig:T_CONV})}
\label{sec:exp:T}
We compare the convergence properties of \CD with different $T_{in}$ values.
As $T_{in}$ increases, \CD tends to converge more stably to better solutions (with lower test RMSE).
However, there is an exception, $T_{in}=1$ in the Netflix${}_{3}$ dataset, which converges to the best solution.

\subsection{Effects of the Number of Columns Updated at a Time (Figure~\ref{fig:C_CONV})} 
\label{sec:exp:C}
We compare the convergence properties of \CA with different $C$ values.
$T_{in}$ is fixed to one in all cases.
As $C$ increases, \CA tends to converge faster to better solutions (with lower test RMSE) although it requires more memory as explained in Theorem~\ref{theorem:memory}.
With $C$ above 20, however, convergence speed starts to decrease.
This trend is related to the running time per iteration, which shows the same trend as seen in Figure~\ref{fig:C_TIME}.
As $C$ increases, the amount of disk I/O declines since it depends on the number of times that the entries of $\tensor{R}$ or $\tensor{\hat{R}}$ are streamed from disk, which is inversely proportional to $C$.
Conversely, computational cost increases quadratically with regard to $C$.
At low $C$ values, the decrease in the amount of disk I/O is greater and leads to a downward trend of running time per iteration. The opposite happens at high $C$ values.

\begin{figure}[bpt!]
	\centering
	\hspace{-0.04\linewidth}
	\subfigure[Netflix${}_3$]{
		\includegraphics[width=0.49\linewidth] {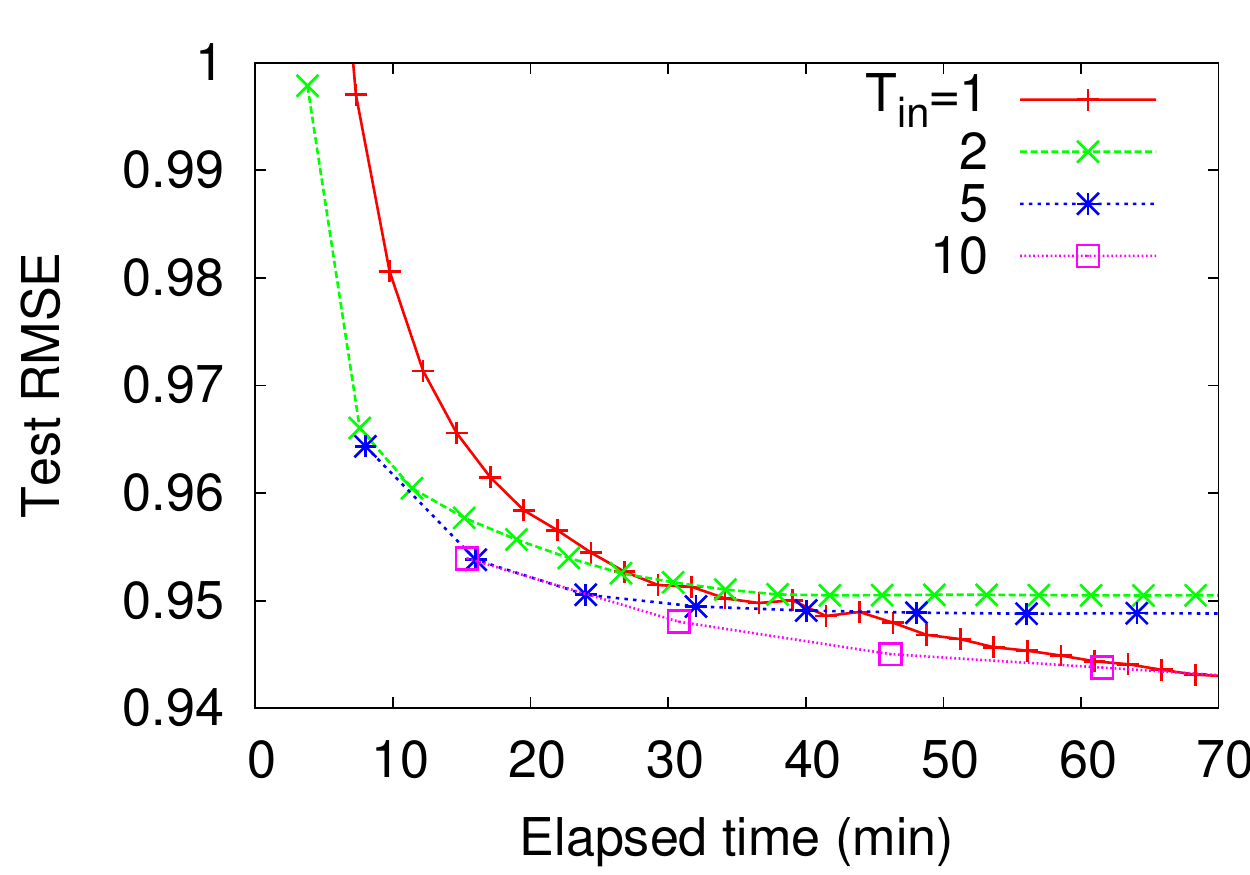}
	}
	\hspace{-0.045\linewidth}
	\subfigure[Yahoo-music${}_4$]{
		\includegraphics[width=0.49\linewidth] {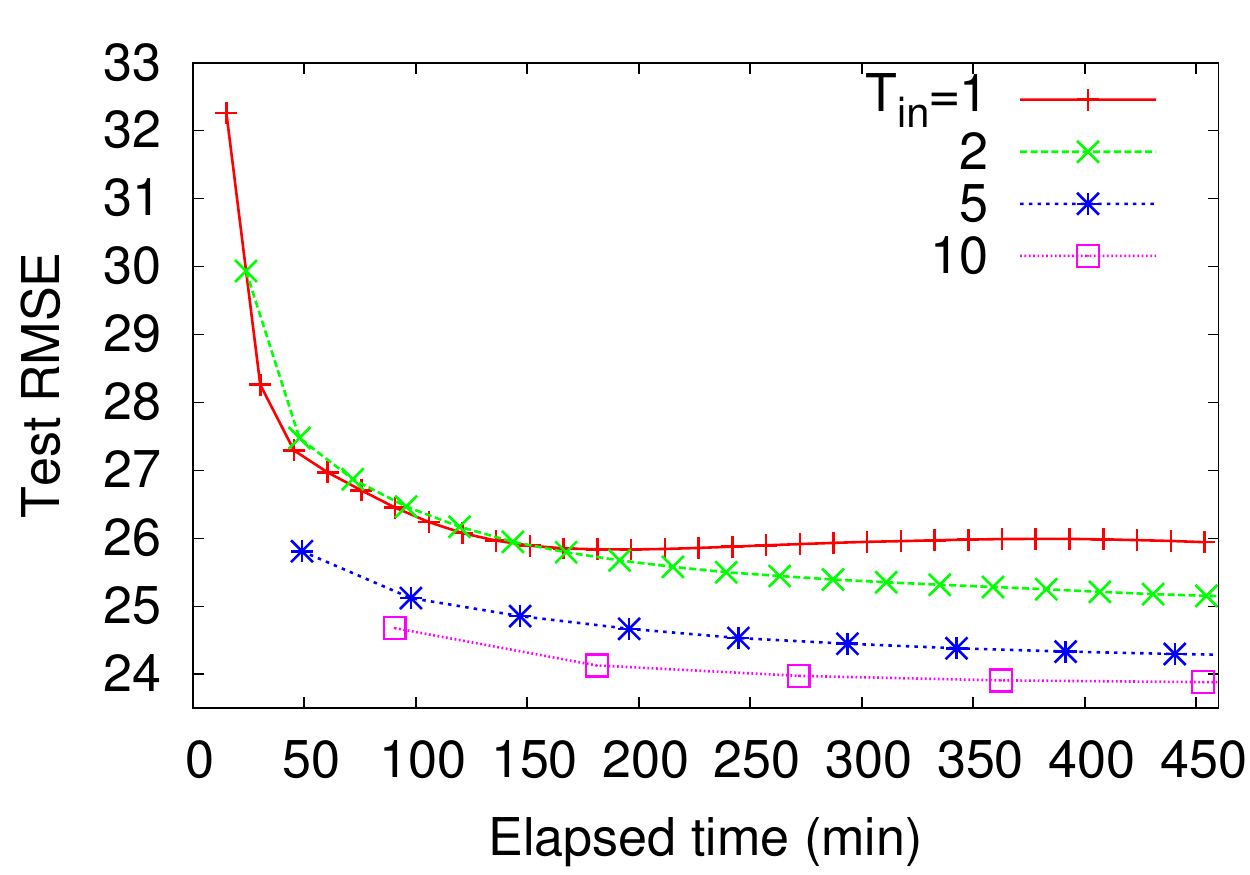}
	}
	\caption{Effects of inner iterations ($T_{in}$) on the convergence of \CD.
		\CD tends to converge stably to better solutions as $T_{in}$ increases.}
	\label{fig:T_CONV}
\end{figure} 

\begin{figure}[bpt!]
	\centering
	\hspace{-0.04\linewidth}
	\subfigure[Netflix${}_3$]{
		\includegraphics[width=0.49\linewidth] {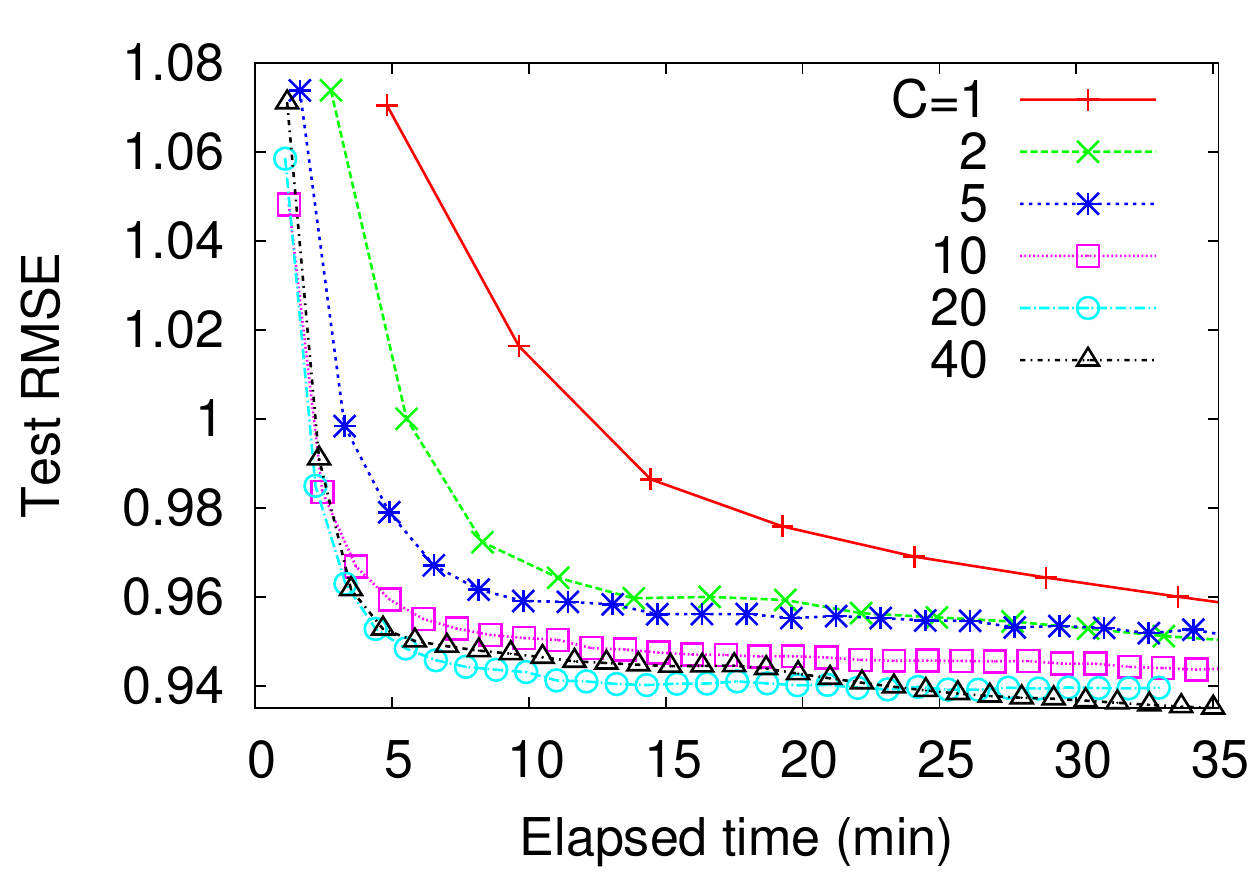}
	}
	\hspace{-0.045\linewidth}
	\subfigure[Yahoo-music${}_4$]{
		\includegraphics[width=0.49\linewidth] {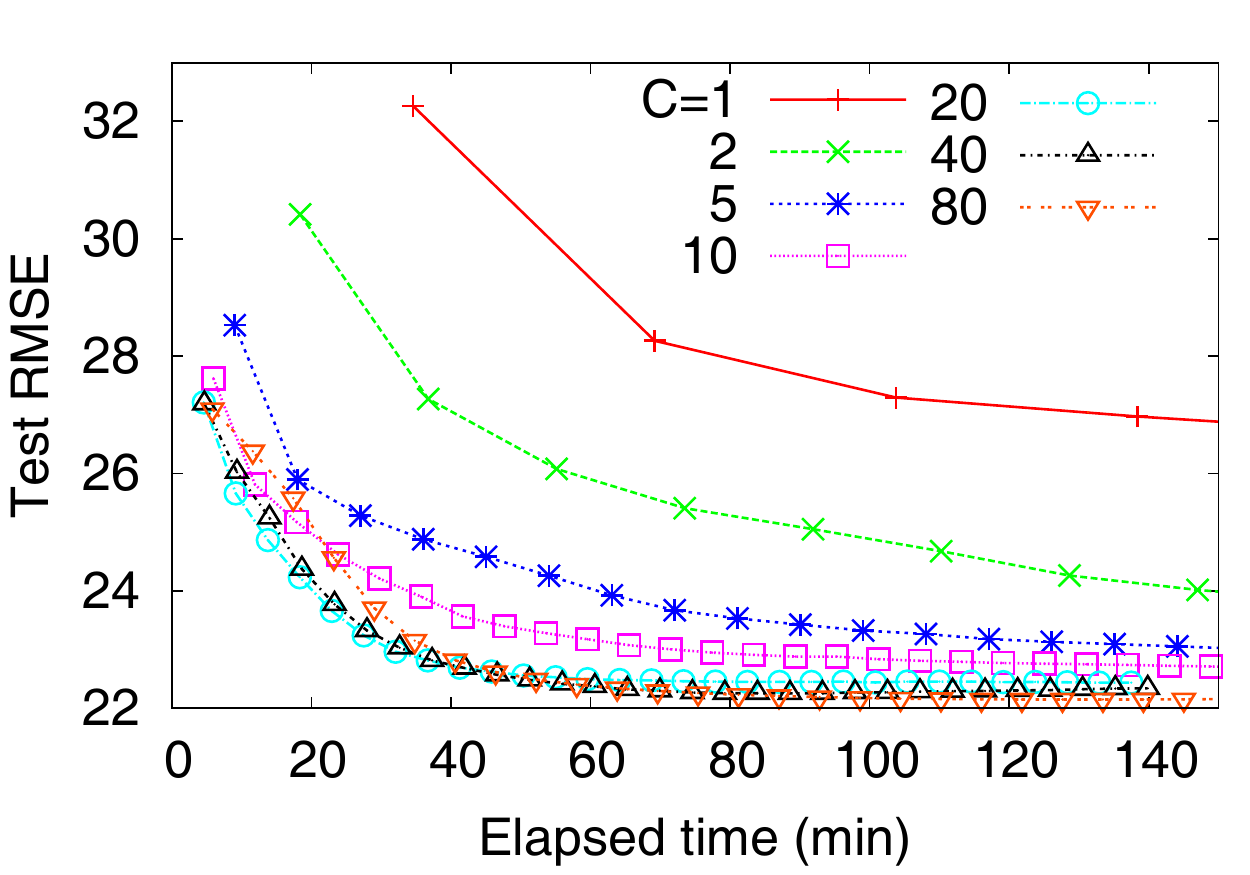}
	}
	\caption{Effects of the number of columns updated at a time ($C$) on the convergence of \CA. \CA tends to converge faster to better solutions as $C$ increases. However, its convergence speed decreases at $C$ above 20.}
	\label{fig:C_CONV}
\end{figure} 

\begin{figure}[bpt!]
	\centering
	\hspace{-0.04\linewidth}
	\subfigure[Netflix${}_3$]{
		\includegraphics[width=0.48\linewidth] {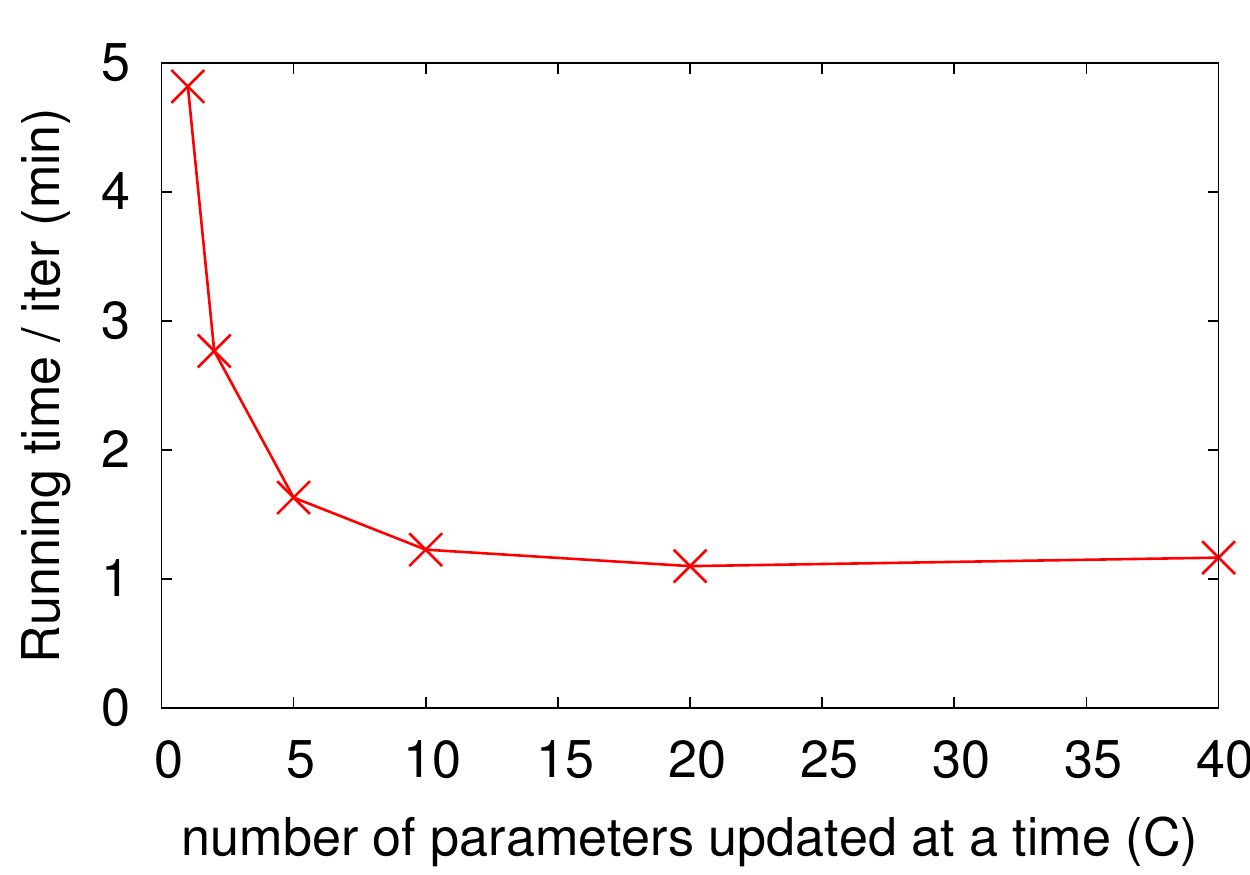}
	}
	\hspace{-0.045\linewidth}
	\subfigure[Yahoo-music${}_4$]{
		\includegraphics[width=0.48\linewidth] {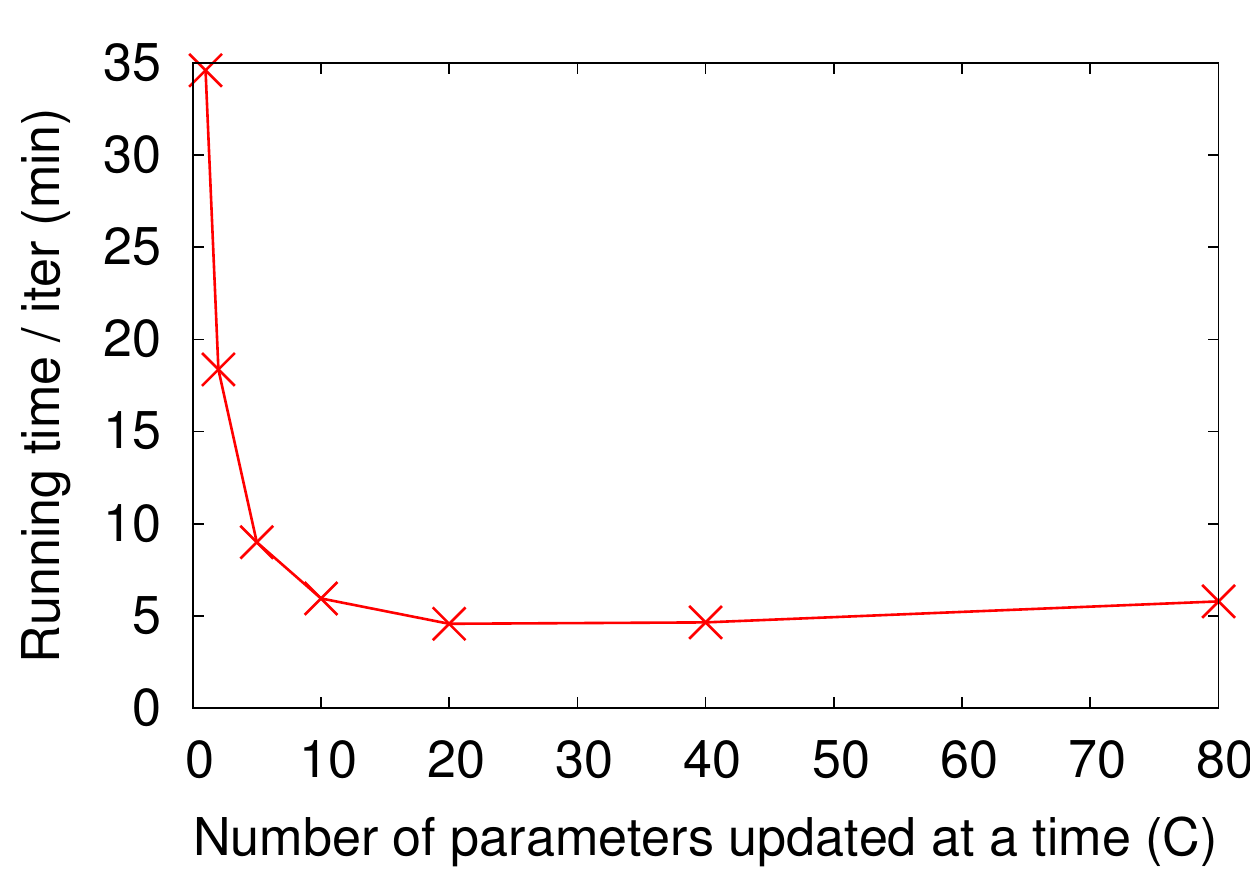}
	}
	\caption{Effects of the number of columns updated at a time ($C$) on the running time of \CA.
		Running time per iteration decreases until $C = 20$, then starts to increase.
	}
	\label{fig:C_TIME}
\end{figure} 

\section{Conclusion}
\label{sec:conclusion}
In this paper, we propose \CA and \CD, distributed algorithms for high-dimensional large-scale tensor factorization.
They are scalable with all aspects of data (i.e., dimension, the number of observable entries, mode length, and rank) and show a trade-off:
\CA has an advantage in terms of convergence speed, and
\CD has one in terms of memory usage.
The local disk caching and the greedy row assignment, two proposed optimization schemes, significantly accelerate not only \CA and \CD but also their competitors.

\section*{Acknowledgments}
This work was supported by AFOSR/AOARD under the Grant No. FA2386-14-1-4036,
and by the National Research Foundation of Korea (NRF) Grant funded by the Korean Government (MSIP) (No. 2013R1A1A1064409)

\bibliography{BIB/myref}
\bibliographystyle{ieeetr}

\clearpage

\end{document}